\documentclass[12pt]{amsart}

\usepackage{geometry}
\usepackage[all]{xy}

\usepackage{cleveref}
\usepackage{epstopdf}
\usepackage{amssymb,amsmath,amsfonts,amsthm,graphicx,enumerate,amscd}
\usepackage{mathrsfs}
\usepackage{fullpage,enumerate}
\usepackage{slashed}

\usepackage{tikz}

\newcommand{\cI}{\mathcal I}

\newcommand{\cK}{\mathcal K}
\newcommand{\cQ}{\mathcal Q}
\newcommand\bbZ{\mathbb Z}
\newcommand\tr{\operatorname{tr}}

\newcommand\CC{\mathbb C}

\theoremstyle{plain}
\newtheorem{theorem}{Theorem}[section]
\newtheorem{conjecture}[theorem]{Conjecture}

\newtheorem{proposition}[theorem]{Proposition}

\theoremstyle{definition}
\newtheorem{example}[theorem]{Example}
\newtheorem{remark}[theorem]{Remark}

\newtheorem{definition}[theorem]{Definition}

\title{The ring structure of twisted equivariant $KK$-theory for noncompact Lie groups}
\author{Chi-Kwong Fok}
{\date{May 19, 2021}}
\address[Chi-Kwong Fok]{
Department of Pure Mathematics,
School of  Mathematical Scienes, 
University of Adelaide, 
Adelaide, SA 5005, 
Australia. Current address: New York University Shanghai, Pudong New Area, Shanghai 200122, China
}
\email{ckfok@nyu.edu}
\author{Varghese Mathai}
\address[Varghese Mathai]{
Department of Pure Mathematics,
School of  Mathematical Scienes, 
University of Adelaide, 
Adelaide, SA 5005, 
Australia}
\email{mathai.varghese@adelaide.edu.au}

\begin{document}
\begin{abstract}
  {Let $G$ be a connected semi-simple Lie group with torsion-free fundamental group.} We show that the twisted equivariant $KK$-theory $KK_{\bullet}^{G}(G/K, \tau_G^G)$ of $G$ has a ring structure
induced from the renowned ring structure of the twisted equivariant $K$-theory $K^{\bullet}_{K}(K, \tau_K^K)$ of a maximal compact subgroup $K$. We give a geometric description of representatives in $KK_{\bullet}^{G}(G/K, \tau_G^G)$ in terms of equivalence classes of certain equivariant correspondences and  {obtain
an optimal set of generators of this ring}. We also
 establish various properties of this ring under some additional hypotheses on $G$ and give an application to the quantization of $q$-Hamiltonian $G$-spaces in an appendix. 
 We also suggest conjectures regarding the relation to positive energy representations of $LG$ that are induced from certain unitary representations of $G$ in the noncompact case.
\end{abstract}

\maketitle

\tableofcontents

\section*{Introduction}
 {Let $K$ be a connected  {compact Lie group with torsion-free fundamental group}, viewed as a $K$-space via the conjugation action. The equivariant cohomology group $H_K^3(K, \mathbb{Z})$ classifies the local coefficient systems called \emph{twists} for the equivariant $K$-theory (or $K$-homology) of $K$. Let $\tau_K^K$ be a twist satisfying a certain positivity condition}. A result of Freed-Hopkins-Teleman \cite{FHT,FHT1} gives the equivariant twisted $K$-homology $K^K_\bullet(K, \tau_K^K)$ of $K$ a ring structure, using the multiplication map of the group $K$. This ring turns out to have a trace, making it a Frobenius ring that has an associated 2 dimensional TQFT \cite{FHT10}, together with numerous other interesting applications to mathematical physics as listed in the introduction of \cite{FHT3}. 

Our main goal in this paper is to generalise this to the noncompact case as follows.  {Let $G$ be a connected semi-simple Lie group with torsion-free fundamental group, and $K$ a maximal compact subgroup}. The equivariant cohomology groups $H_G^\bullet(G, \mathbb{Z})$ are defined to be the cohomology groups of the Borel construction $EG\times_G G$, which is homotopy equivalent to $EK\times_K K$ which is the Borel construction of $K$. In particular,  
$H_G^3(G, \mathbb{Z})$ ($\cong H_K^3(K, \mathbb{Z})$) classifies the local coefficient systems called \emph{equivariant twists} for the equivariant $K$-theory (or $K$-homology) of $G$. Let $\tau_G^G$ be an equivariant twist whose cohomology class is in 
$H_G^3(G, \mathbb{Z})$. Our first main result  {Theorem \ref{mainthm}} (\ref{ringstructure}) says that the equivariant $KK$-theory $KK^G_\bullet(G/K, (G, \tau_G^G))$ has a natural  ring structure that is induced from the ring structure of $K^K_\bullet(K, \tau_K^K)$, which we recall is obtained using the multiplication map on $K$. The key difficulty in the noncompact case is that the multiplication map on $G$ is {\em not} a proper map, so it does {\em not} induce a map on (equivariant) twisted $K$-homology or on (equivariant) twisted $K$-theory.  {Note that, if $G$ itself is compact, then $G=K$ and Theorem \ref{mainthm} (\ref{ringstructure}) reduces to the ring $K^K_\bullet(K, \tau_K^K)$ in the Freed-Hopkins-Teleman Theorem.}

In \S \ref{sect:correspondences}, we give a sketch of a geometric description of the elements of  {equivariant $KK$-theory} in terms of certain equivariant correspondences, generalising
to the twisted case {by} earlier work of others \cite{ConnesSkandalis84,EM}. 
In the case when  {$G$ is connected, simply-connected and semi-simple}, we show in \S \ref{sect:generators} that 
there is a natural set of generating correspondences in $KK^G_\bullet(G/K, (G, \tau_G^G))$, that are derived from the generating set in the case of the compact subgroup $K$ via the isomorphisms  
as in \S \ref{sect:generators}. This is the second key theorem in our paper.

Under the same assumptions as above, there is a commutative diagram, 
\begin{eqnarray}\label{commutativediagram}
			\xymatrix{R(K)\ar[r]^{{\iota_!^K}}\ar[d]^{\text{D-Ind}}& K_\bullet^K(K, \tau_K^K)\ar[d]^{\phi}\\ K_\bullet(C^*_rG)\ar[r]_{{\iota_!^G}\qquad}& KK^G_\bullet(G/K, (G, \tau_G^G))}.
		\end{eqnarray}
		 {Here $\phi$  consists of a composition of natural isomorphisms as in the proof of Theorem \ref{mainthm} (\ref{ringstructure}) as follows.}
\begin{enumerate}
\item The first is {\em equivariant Poincar\'e duality} isomorphism, $K^K_\bullet(K, \tau_K^K) \cong K_K^{{\bullet+\text{dim }K}}(K, \tau_K^K)$.

\item The {\em Green-Julg} isomorphism, $K_K^\bullet(K, \tau_K^K) \cong K_\bullet(\Gamma_0(\tau_K^K)\rtimes K)$.		
\item  {The restriction map $KK^G_\bullet(G/K, (G, \tau_K^G))\to KK_\bullet^K(G/K, (G, \tau_G^G))$ (cf. \cite[Theorem 5.8]{Ka}).}
\end{enumerate}	
		We define ${\iota_!^G} = \phi \circ {\iota_!^K} \circ (\text{D-Ind})^{-1}$ (there is an
		alternate direct definition of ${\iota_!^G}$, but it is not needed here).
Since Dirac induction $\text{D-Ind}$ is an isomorphism (\cite{CEN}) and $R(K)$ is a ring, it follows that $K_*(C^*_rG)$ is also a ring.
Also $\phi$ is an isomorphism by our main result, and ${\iota_!^K}$ is surjective under the hypotheses on $G$ (hence on $K$)
\cite{FHT}, so we deduce that  ${\iota_!^G}$ is 
also surjective. 

It follows that if $\cI_K$ denotes the Verlinde ideal for $K$, then $$K_\bullet^K(K, \tau_K^K) \cong R(K)/\cI_K,$$ then it is natural to define {$\cI_G =  \text{D-Ind} (\cI_K)$} as the Verlinde ideal for 
$G$, in which case $$KK^G_\bullet(G/K, (G, \tau_G^G)) \cong K_\bullet(C^*_rG)/\cI_G,$$
showing that $KK^G_\bullet(G/K, (G, \tau_G^G))$ is a {\em Verlinde ring}.
Now $ K_0(C^*_rG)$  is a ring by Dirac induction, so this gives an alternate description of the ring structure of $KK^G_\bullet(G/K, (G, \tau_G^G)) $.
We note that 
there is an induced trace 
on $KK^G_0(G/K, (G, \tau_G^G))$, thus showing that it is a finite dimensional {\em Frobenius ring}.

 {Other main results of ours include}
\begin{enumerate}

\item\label{CK} the {\em Connes-Kasparov} isomorphism in this context, $$K_\bullet(\Gamma_0(\tau_K^K)\rtimes K) \cong 
K_\bullet(\Gamma_0(\tau_G^G)\rtimes G),$$ and
\item\label{BC}  the  {\em Baum-Connes} conjecture in this context, $$KK^G_\bullet(G/K, (G, \tau_G^G)) \cong K_\bullet(\Gamma_0(\tau_G^G)\rtimes G).$$
\end{enumerate}

The sketch of the proofs of (\ref{CK}) and (\ref{BC}) are as follows. By \S7 in \cite{CEN}, one has $K_\bullet(\cK\rtimes K) \cong 
K_\bullet(\cK\rtimes G)$ when $\dim(G/K)$ is even, which has to be slightly modified when $\dim(G/K)$ is odd. Next in \S \ref{noncompact}
we use the
Steinberg map on $G$ to define an interesting {closed} cover of $G$ that is invariant under the adjoint action of $G$ and so are their intersections. 
Moreover {there are $G$-equivariant proper deformation retracts of both the closed sets and their intersections onto either a stratification of conjugacy classes of $G$ or a fiber bundle over a closed cone. The Baum-Connes and Connes-Kasparov conjectures are then shown to hold for these closed sets.} Then we can use the Mayer-Vietoris sequence 
(in the noncompact case) to prove  {(\ref{CK}). (\ref{BC})} is proved analogously.


When $G$ is semisimple, it is well known that $ K_0(C^*_rG)$ is generated by discrete series representations of $G$ and limits of discrete series representations of $G$ in the equirank case cf. \cite{BCH}. Inspired by the compact case
(i.e. the Freed-Hopkins-Teleman theorem \cite{FHT}), 
it suggests a conjecture that there are isomorphisms between 
$KK^G_\bullet(G/K, (G, \tau_G^G))$,  {$K_\bullet(\Gamma_0(\tau_G^G)\rtimes G)$} and certain positive energy representations of $LG$ that are induced from discrete series representations of $G$ and their limits, in the equirank case. In the case when $G$ is a complex semisimple Lie group, then there are no discrete series representations of $G$, but the $ K_\bullet(C^*_rG)$ is generated instead by the unitary principal-series representations of $G$ \cite{PP}, so it suggests a conjecture that there are isomorphisms between 
$KK^G_\bullet(G/K, (G, \tau_G^G))$,  {$K_\bullet(\Gamma_0(\tau_G^G)\rtimes G)$} and certain positive energy representations of $LG$ induced from the unitary principal series representations of $G$.  {In Appendix {{\ref{kcstar}}}, we review the $K$-theory of the reduced $C^*$-algebra of $G$ and Connes-Kasparov conjecture in the real reductive linear and complex semi-simple cases (\cite{W, CCH, PP}). Based on these results we describe $K_\bullet(\Gamma_0(\tau_G^G)\rtimes G)$ for $G$ complex semi-simple in terms of weights, which we hope will prove useful in the computational aspect of relevant future research in mathematical physics, and propose our conjecture on the equivariant twisted $K$-theory of $G$ and certain positive energy representations of $LG$.} 

In Appendix B, we study $q$-Hamiltonian induction and cross-section in the noncompact case: this was already established in the compact case in \cite{AMM}. In particular, we show that the set of generators of $KK^G_0(G/K, (G, \tau_G^G))$ obtained in section 3, are all $q$-Hamiltonian $G$-spaces, where $G$ acts properly and cocompactly, and we define the quantization of such $q$-Hamiltonian $G$-spaces, generalizing the definition of Meinrenken \cite{M2} in the compact case.

\bigskip

\noindent\textbf{Acknowledgements}. The authors thank Jonathan Rosenberg, James Humphreys and a member of MathOverflow (username: Dap) for their advice on finding a {suitable} cover as in Proposition \ref{opencover}, {and the anonymous referee for critical comments and in particular pointing out a mistake in an earlier draft of the paper}.   VM thanks the participants of the conference,{\em Quantization in symplectic geometry}, Cologne, July 15 - 19, 2019, for feedback on his talk over there.
Both authors acknowledge support from the Australian Research Council 
through the Discovery Project grant DP150100008.
Varghese Mathai thanks the Australian Research Council  for support via the Australian Laureate Fellowship FL170100020.

\section{Review of the compact Lie group case}
 {Let $K$ be a connected and simply-connected compact Lie group, viewed as a $K$-space via the conjugation action.} By the structure theorem for compact Lie groups, $K$ is a product of simple, simply-connected compact Lie groups $K_i$, $1\leq i\leq n$. The equivariant cohomology group $H_K^3(K, \mathbb{Z})$ classifies the local coefficient systems called \emph{twists} for the equivariant $K$-theory (or $K$-homology) of $K$, as well as the central extensions of the loop group $LK$. It is known that $H_K^3(K, \mathbb{Z})\cong \bigoplus_{i=1}^nH_{K_i}^3(K_i, \mathbb{Z})$, where $H_{K_i}^3(K_i, \mathbb{Z})\cong\mathbb{Z}$. Let $\tau_K^K$ be a twist whose cohomology class is $\vec{k}:=(k_1, k_2, \cdots, k_n)\in H_K^3(K, \mathbb{Z})$. The adjoint map $\text{Ad}: K\to SO(\mathfrak{k})$ induces the map $\text{Ad}^*: H_{SO(\mathfrak{k})}^3(SO(\mathfrak{k}), \bbZ)\to H_K^3(K, \mathbb{Z})$, and we call the image of $(1, 1)\in H_{SO(\mathfrak{k})}^3(SO(\mathfrak{k}), \mathbb{Z})\cong\mathbb{Z}\oplus\mathbb{Z}_2$ or the twist over $K$ representing  {it} the \emph{adjoint shift}, denoted by $\textsf{h}^\vee$. A theorem of Freed-Hokins-Teleman \cite{FHT}
gives a ring structure to the equivariant twisted $K$-theory of $K$, $K_K^*(K, \tau_K^K)$.
\begin{theorem}[Freed-Hopkins-Teleman]
\begin{enumerate}
	\item We have {that}
	\[K_K^\bullet(K, \tau_K^K)\] is a ring, with Pontryagin product induced by the multiplication map on the compact group $K$ and equivariant Poincar\'e duality in equivariant twisted K-theory.
	\item The map
	\[{\iota_!}: R(K)\cong K_K^\bullet(\text{pt})\to K_K^\bullet(K, \tau_K^K)\]
	given by the wrong-way map induced by the inclusion of the group identity into $K$ is onto, and its kernel is precisely the level $\vec{k}-\textsf{h}^\vee$ Verlinde ideal $\cI_K$
	of $R(K)$. 
\end{enumerate}
\end{theorem}

\begin{remark}
The main Freed-Hopkins-Teleman Theorem identifies the equivariant twisted $K$-theory of $K$ with the Verlinde ring of its loop group $LK$ (i.e. the representation group of positive energy representations of $LK$ with fusion product structure). We will not use this in our paper.
\end{remark}

\section{The noncompact Lie group case}\label{noncompact}

It would be interesting to establish a version of the Freed-Hopkins-Teleman Theorem for noncompact Lie group $G$. Here we would like to achieve the modest goal of explicitly describing the equivariant twisted $K$-{theory} of $G$ by relating it to that of the maximal compact subgroup $K$. 

\subsection{The twists}\label{thetwists}  {In this section, we construct equivariant twists on $G$. For convenience of exposition we shall assume that $G$ is a connected, simply-connected and simple Lie group. One can easily adapt the construction in this section to the case where $G$ is connected semi-simple Lie group with torsion-free fundamental group}.

The Cartan decomposition $\mathfrak{g}=\mathfrak{k}\oplus\mathfrak{p}$ induces the diffeomorphism
\begin{align*}
	m: K\times \mathfrak{p}&\to G\\
	(k, \zeta)&\mapsto k\text{exp}(\zeta),
\end{align*}
where $\text{exp}(\mathfrak{p})$ is contractible. Noting that $[\mathfrak{k}, \mathfrak{p}]\subseteq \mathfrak{p}$, we have that the map
\begin{align*}
	F: G\times I&\to G\\
	(k\text{exp}(\zeta), t)&\mapsto k\text{exp}(t\zeta)
\end{align*}
witnesses a $K$-equivariant deformation retract from $G$ onto $K$. Moreover, $BK\simeq EG/K$ can be realized as a $G/K$-bundle over $BG=EG/G$, and $G/K$ is contractible. Thus $BK$ and $BG$ are homotopy equivalent. It follows that the restriction maps
\[H_G^3(G, \mathbb{Z})\longrightarrow H_K^3(G, \mathbb{Z})\longrightarrow H_K^3(K, \mathbb{Z})\cong\mathbb{Z}\]
are isomorphisms. Let $[\tau_K^G]\in H_K^3(G, \mathbb{Z})$ and $[\tau_G^G]\in H_G^3(G, \mathbb{Z})$ be the preimages of $[\tau_K^K]$ under the above restriction maps. 
Now we would like to construct a DD-bundle over $G$ whose cohomology class is the positive generator of $H_G^3(G, \mathbb{Z})$. If $G$ itself is a complex simple Lie group, let $\mathcal{H}_K$ be the basic positive energy representation of $\widetilde{LK}$ which is the level 1 central extension of $LK$ (see \cite[\S 4]{Se}). Note that the central circle of $\widetilde{LK}$ acts on $\mathcal{H}_K$ with weight 1. This representation can be extended to a holomorphic representation of $\widetilde{LG}$ (\cite[Corollary 4.3 (b)]{Se}). Let $P_eG$ be the based path group and $L_eG$ the based loop group. Then $P_eG\times_{L_eG}\mathcal{K}(\mathcal{H}_K)\to G$ is the desired $G$-equivariant DD-bundle as it restricts to $P_eK\times_{L_eK}\mathcal{K}(\mathcal{H}_K)$, which is a fundamental $K$-equivariant DD-bundle over $K$ (\cite[\S 3, first paragraph]{M1}). If $G$ is a general real simple Lie group, then consider its complexification $G^\mathbb{C}$ and its maximal compact subgroup $Q$. We can reprise the trick above and let $\mathcal{H}_Q$ be the basic positive energy representation of $\widetilde{LQ}$, which can be extended to a holomorphic representation of $\widetilde{LG^\mathbb{C}}$ and further restricted to $\widetilde{LG}$. The $G$-equivariant DD-bundle $P_eG\times_{L_eG}\mathcal{K}(\mathcal{H}_Q)$ is the desired fundamental $G$-equivariant DD-bundle over $G$. 

Let $\tau_K^K:=P_eK\times_{L_eK}\mathcal{K}(\mathcal{H}_K^{\otimes k})$, $\tau_G^G:=P_eG\times_{L_eG}\mathcal{K}(\mathcal{H}_Q^{\otimes k})$ and $\tau_K^G$ the same as $\tau_G^G$ but with restricted $K$-action. {Let 
\[\text{th}: K_K^\bullet(K, \tau_K^K)\to K_K^{\bullet+\text{dim }G-\text{dim }K}(K\times\mathfrak{p}, m^*\tau_K^G)\]
be the Thom isomorphism (note that $K\times\mathfrak{p}$ is $K$-theoretic orientable $K$-equivariant vector bundle over $K$ as $H_K^3(K, \mathbb{Z})$ does not have 2-torsion). Define the map 
\[i^*: K_K^\bullet(G, \tau_K^G)\to K^{\bullet+\text{dim }K-\text{dim }G}_K(K, \tau_K^K)\]
to be the composition $\text{th}^{-1}\circ m^*$, which can be seen to be an isomorphism.} By the Green-Julg Theorem, the $K$-theory $K_K^\bullet(G, \tau_K^G)$ is the $K$-theory of the crossed product $C^*$-algebra $\Gamma_0(\tau_K^G)\rtimes_\alpha K$, where $\Gamma_0(\tau_K^G)$ is the $C^*$-algebra of sections of $\tau_K^G$ vanishing at infinity, and $\alpha$ is the action map of $K$ on $\Gamma_0(\tau_K^G)$. We will omit $\alpha$ from the notation if there is no confusion about the actions involved.

\subsection{The Baum-Connes assembly map}\label{ringstr}

 {The Baum-Connes map in our context is the map
$$
\mu: KK^G_\bullet(G/K, (G, \tau_G^G))\longrightarrow KK_\bullet(\CC, \Gamma_0(G, \tau_G^G)\rtimes G)=K_\bullet(\Gamma_0(G, \tau_G^G)\rtimes G),
$$
which is defined as follows. It is defined as the composition of a couple of maps, the first of which is Kasparov's induction to the crossed product
$$
KK^G_\bullet(G/K, (G, \tau_G^G))\longrightarrow KK_\bullet(C_0(G/K)\rtimes G, \Gamma_0(G, \tau_G^G)\rtimes G).
$$
Before defining the second map, since $G/K$ is a proper $G$-space that is $G$-cocompact, there is a cutoff function $c$
that satisfies 
$$
\int_G c^2(g.x) dg =1, \qquad \forall x\in G/K
$$
Using $c$, define a projection $p$ in $C_0(G/K)\rtimes G$ as follows. $p(x, g) := c(x)c(g.x).$ Then $p$ determines an element $[p] \in K_0(C_0(G/K)\rtimes G)$.
The second map is then defined as the Kasparov intersection product with $[p]$, that is
$$
[p] \otimes_{C_0(G/K)\rtimes G} -: KK_\bullet(C_0(G/K)\rtimes G, \Gamma_0(G, \tau_G^G)\rtimes G) \longrightarrow KK_\bullet(\CC, \Gamma_0(G, \tau_G^G)\rtimes G).
$$
We will show that $\mu$ is an isomorphism, and since $KK_\bullet(\CC, \Gamma_0(G, \tau_G^G)\rtimes G)= K_\bullet(\Gamma_0(G, \tau_G^G)\rtimes G)$ is a ring, with ring structure 
induced from  $K^{\bullet}_{K}(K, \tau_K^K)$,
it follows that  $KK^G_\bullet(G/K, (G, \tau_G^G))$ is also a ring, with ring structure induced from $K^{\bullet}_{K}(K, \tau_K^K)$.}

\subsection{A commutative diagram}
 {In this section, we would like to define the maps in, and show the commutativity of, the following diagram
\begin{eqnarray}\label{commdiag2}
	\xymatrix{K_K^\bullet(\text{pt})\ar[r]^{\text{D-Ind}}\ar[d]_{{\iota_!^K}}&K_\bullet(C_r^*G)\ar[d]^{{\iota_!^G}}\\ K_K^\bullet(K, \tau_K^K)\ar[r]_q& K_\bullet(\Gamma_0(\tau_G^G)\rtimes G)}
\end{eqnarray}
The map D-Ind is the Dirac induction in the Connes-Kasparov conjecture, which in order to facilitate the proof of commutativity will be redefined by means of the Baum-Connes map as follows. 
\[KK^K_\bullet(\text{pt}, \text{pt})\stackrel{[G/K]\otimes\cdot}{\longrightarrow} KK^K_\bullet(G/K, \text{pt})\stackrel{(\text{res}_G^K)^{-1}}{\longrightarrow} KK^G_\bullet(G/K, \text{pt})\stackrel{\text{BC}}{\longrightarrow} K_\bullet(C_r^*G)\]
Here $[G/K]\in KK^K_\bullet(G/K, \text{pt})$ is the fundamental class of $G/K$ (the Dirac element for $G/K$), and by \cite[Theorem 5.8]{Ka} (phrased as \cite[Proposition 3.2]{MN}), the restriction map is an isomorphism. The map $q$ can be defined similarly as follows. 
\begin{align*}\label{commutativediagram}
	KK^K_\bullet(\text{pt}, (K, \tau_K^K))&\stackrel{[G/K]\otimes\cdot}{\longrightarrow} KK^K_\bullet(G/K, (K, \tau_K^K))\\
					&\stackrel{ {\cdot\otimes[i]}}{\longrightarrow} KK^K_\bullet(G/K, (G, \tau_K^G))\\
					&\stackrel{(\text{res}_G^K)^{-1}}{\longrightarrow} KK^G_\bullet(G/K, (G, \tau_G^G))\\
					&\stackrel{\text{BC}}{\longrightarrow} K_\bullet(\Gamma_0(\tau_G^G)\rtimes G).
\end{align*}
 {Here $[i]$ is the element in $KK^K_\bullet((K, \tau_K^K), (G, \tau_K^G))$ induced by the equivariant $K$-orientable map $i: K\to G$ which is the inclusion map, and intersection product with $[i]$ amounts to the Thom isomorphism.} The map ${\iota_!^K}$ is the pushforward induced by the inclusion of the point as the identity element in $K$. One may interpret this map as the Kasparov product with $1_K$, which is the identity element of the ring $K_K^\bullet(K, \tau_K^K)=KK_\bullet^K(\text{pt}, (K, \tau_K^K))$. The map ${\iota_!^G}$ can be defined similarly using Kasparov product as well. Consider the Kasparov product
\[KK^G_\bullet(\text{pt}, G/K)\otimes KK^G_\bullet(G/K, (G, \tau_G^G))\to KK^G_\bullet(\text{pt}, (G, \tau_G^G))\]
and the Kasparov descent map
\[d: KK^G(\text{pt}, (G, \tau_G^G))\to KK_\bullet(C_r^*G, \Gamma_0(\tau_G^G)\rtimes G).\]
Let $1_{G/K}$ be the identity element of $KK^G_\bullet(\text{pt}, G/K)$ identified with the representation ring $R(K)$ (in fact $1_{G/K}$ is the dual Dirac element of $KK^G_\bullet(\text{pt}, G/K)$).  {Let $1_{G/K, G}$ be $[G/K]\otimes_{C(\text{pt})}1_K\otimes_{\Gamma(\tau_K^K)}[i]\in KK_\bullet^G(G/K, (G, \tau_G^G))$, which is identified with the ring $K^\bullet_K(K, \tau_K^K)$ through the map $[G/K]\otimes_{C(\text{pt})}\cdot\otimes_{\Gamma(\tau_K^K)}[i]$. }
Define $1_G:=1_{G/K}\otimes_{C_0(G/K)}1_{G/K, G}$, $1_{G, d}:=d(1_G)$ and the map $ {\iota_!^G}$ to be the Kasparov product with $1_{G, d}$. To show commutativity of the diagram (\ref{commutativediagram}), we break it up into four diagrams and prove each of them is commutative.
\begin{enumerate}
	\item First we have 
	\begin{eqnarray}
		\xymatrix{KK^K_\bullet(\text{pt}, \text{pt})\ar[r]^{[G/K]\otimes\cdot}\ar[d]^{\cdot\otimes 1_K}& KK^K_\bullet(G/K, \text{pt})\ar[d]^{\cdot\otimes 1_K}\\ KK^K_\bullet(\text{pt}, (K, \tau_K^K))\ar[r]^{[G/K]\otimes\cdot}& KK^K_\bullet(G/K, (K, \tau_K^K))}
	\end{eqnarray}
which is commutative because of the associativity of Kasparov product. 
	\item Next we have the diagram 
	\begin{eqnarray}
		\xymatrix{KK^K_\bullet(G/K, \text{pt})\ar[r]^=\ar[d]^{\cdot\otimes 1_K}& KK^K_\bullet(G/K, \text{pt})\ar[d]^{\cdot\otimes 1_{G, K}}\\ KK^K_\bullet(G/K, (K, \tau_K^K))\ar[r]^{ {\cdot\otimes[i]}}& KK^K_\bullet(G/K, (G, \tau_K^G))}
	\end{eqnarray}
	where $1_{G, K}\in KK^K_\bullet(\text{pt}, (G, \tau_K^G))$ is defined to be $ {1_K\otimes_{\Gamma(\tau_K^K)} [i]}$. Note that $1_{G, K}$ is the identity element of $KK^K_\bullet(\text{pt}, (G, \tau_K^G))$ which is identified with the ring $KK^K_\bullet(\text{pt}, (K, \tau_K^K))=K_K^\bullet(K, \tau_K^K)$  {through intersection product with $[i]$}.
	\item The diagram
	\begin{eqnarray}
		\xymatrix{KK^K_\bullet(G/K, \text{pt})\ar[d]^{\cdot\otimes 1_{G, K}}&\ar[l]_{\text{res}_G^K} KK^G_\bullet(G/K, \text{pt})\ar[d]^{\cdot\otimes 1_G}\\ KK^K_\bullet(G/K, (G, \tau_K^G))&\ar[l]_{\text{res}_G^K} KK^G_\bullet(G/K, (G, \tau_G^G))}
	\end{eqnarray}
	is commutative because $\text{res}_G^K(1_G)=1_{G, K}$. In more details, we have 
	\begin{align*}
		\text{res}_G^K(1_G)&=\text{res}_G^K(1_{G/K}\otimes_{C_0(G/K)}1_{G/K, G})\\
		&=\text{res}_G^K(1_{G/K})\otimes_{C_0(G/K)}\text{res}_G^K(1_{G/K, G})	
	\end{align*}
	where $\text{res}_G^K(1_{G/K})$ is the identity element of $KK^K_\bullet(\text{pt}, G/K)$ which is identified with the representation ring $R(K)=K_K^\bullet(\text{pt})$ and $\text{res}_G^K(1_{G/K, G})$ is the identity element of $KK^K_\bullet(G/K, (G, \tau_K^G))$ which is identified with $K_K^\bullet(K, \tau_K^K)$. The following diagram
	\begin{eqnarray}
		\xymatrix{KK^K_\bullet(\text{pt}, G/K)&\otimes KK^K_\bullet(G/K, (G, \tau_K^G))\ar[r]&KK^K_\bullet(\text{pt}, (G, \tau_K^G))\\ K_K^\bullet(\text{pt})\ar[u]_{\cdot\otimes\text{res}_G^K(1_{G/K})}&\otimes K_K^\bullet(K, \tau_K^K)\ar[u]_{ {[G/K]\otimes\cdot\otimes[i]}}\ar[r]& K_K^\bullet(K, \tau_K^K)\ar[u]_{ {\cdot\otimes [i]}}}
	\end{eqnarray}
	is commutative because $\text{res}_G^K(1_{G/K})$ is the dual Dirac element of $KK^K_\bullet(\text{pt}, G/K)$, $[G/K]$ is the Dirac element of $KK^K(G/K, \text{pt})$, and $\text{res}_G^K(1_{G/K})\otimes_{C_0(G/K)} {[G/K]}$ is the identity in $KK^K_\bullet(\text{pt}, \text{pt})$ since the group $K$ acting on the $KK$-theory is amenable. Thus the Kasparov product in the first row of the above diagram can be identified with the Kasparov product in the second row, which amounts to the action of $K_K^*(\text{pt})$ on $K^*_K(K, \tau_K^K)$ through the $K_K^*(\text{pt})$-module structure of the latter. Finally, 
		 {\begin{align*}
			\text{res}_G^K(1_{G/K})\otimes_{C_0(G/K)}\text{res}_G^K(1_{G/K, G})&=\text{res}_G^K(1_{G/K})\otimes_{C_0(G/K)}[G/K]\otimes_{C(\text{pt})}1_K\otimes_{\Gamma(\tau_K^K)}[i]\\
																	&=1_K\otimes_{\Gamma(\tau_K^K)}[i]\\
																	&=1_{G, K}
		\end{align*}}
		as claimed. 
	\item The last diagram we need to prove to be commutative is 
	\begin{eqnarray}
		\xymatrix{KK^G_\bullet(G/K, \text{pt})\ar[r]^{\text{BC}}\ar[d]^{\cdot\otimes 1_G}&K_\bullet(C_r^*G)\ar[d]^{\cdot\otimes1_{G, d}}\\ KK^G_\bullet(G/K, (G, \tau_G^G))\ar[r]^{\text{BC}}& K_\bullet(\Gamma_0(\tau_G^G)\rtimes G)}
	\end{eqnarray}
	and this can be seen by the definition of the Baum-Connes map given in Section \ref{ringstructure}, and associativity of the Kasparov product, and that $d(1_G)=1_{G, d}$ by definition. 
\end{enumerate}}

\subsection{A spectral sequence argument}
In this section we recall Segal's spectral sequence argument used in \cite{M1} to compute the equivariant twisted $K$-theory of a simply-connected simple compact Lie group $K$. Note that Segal's spectral sequence is best suited to computing the generalized (co)homology of simplicial spaces and a generalization of Mayer-Vietoris sequence applied to a certain open cover specified by the simplicial structure. In \cite{M1} the spectral sequence is applied to equivariant twisted $K$-homology. Here we shall describe this argument by adapting the spectral sequence to equivariant twisted $K$-theory which involves nothing other than reversing the arrows of the differentials. 

Fix a maximal torus $T$ and a positive root system. Let $\Delta\subset\mathfrak{t}$ be the Weyl alcove with respect to the above choices. We label the vertices of $\Delta$ by the index set $\{0, 1, \cdots, n\}$, where $n=\text{rank }K$. Let $I\subseteq \{0, 1, \cdots, n\}$. Let $\Delta_I$ be the subsimplex spanned by the vertices with labels in $I$. Then $K$ admits the following simplicial description
\[K=\coprod_I K/K_I\times \Delta_I/\sim\]
where $K_I$ is the stabilizer subgroup of the elements in $\Delta_I$ and $\sim$ is the equivalence relation
\[(kK_I, \iota(t))\sim(kK_J, t)\]
for $J\subseteq I$ and the inclusion map $\iota: \Delta_J\to \Delta_I$. One can construct $\tau_K^K$ by the simplicial approach which is the `slow paced' version of the path space construction (\cite[\S 3]{M1}). Briefly, the $U(1)$-central extension $\widehat{K}_I$ of $K_I$ can be constructed using the basic inner product on $\mathfrak{k}$. By \cite[Lemma 3.6]{M1}, there exists a Hilbert space $\mathcal{H}$ (which in fact is the basic representation of $LK$) which is unitary representations of $\widehat{K}_I$ such that the central circle acts with weight $-1$ and the representation of $\widehat{K}_J$ restricts to that of $\widehat{K}_I$ for $J\subseteq I$. Let $\mathcal{A}_I$ be $K\times_{K_I}\mathcal{K}(\mathcal{H})$. Now $\tau_K^K$ can be defined as 
\[\tau_K^K=\coprod_I\mathcal{A}_I\times\Delta_I/\sim.\]
With the simplicial descriptions of both $K$ and the twist $\tau_K^K$, we can apply Segal's spectral sequence to $K_K^*(K, (\tau_K^K)^{k+\textsf{h}^\vee})$, where the $E_1^{p, q}$-page is 
\begin{align*}
	E_1^{p, q}&=\bigoplus_{|I|=p+1}K_K^{p+q}(K/K_I\times \Delta_I, K/K_I\times\partial\Delta_I, \mathcal{A}_I^{k+\textsf{h}^\vee}\times\Delta_I)\\
			&\cong\bigoplus_{|I|=p+1}K_K^q(K/K_I, \mathcal{A}_I^{k+\textsf{h}^\vee})\\
			&\cong\begin{cases}\bigoplus_{|I|=p+1} R(\widehat{K}_I)_k&\text{ if }q{\text{ is even}}\\0&{\text{ if }q\text{ is odd}},\end{cases}
\end{align*}
and $R(\widehat{K}_I)_k$ is the  {subgroup} of representations of $\widehat{K}_I$ such that the central circle acts with weight $-k$. The whole $E_1$-page is the following sequence of $R(K)$-modules
\[0{\longrightarrow}\bigoplus_{|I|=1}R(\widehat{K}_I)_k\stackrel{\partial_1}{\longrightarrow}\bigoplus_{|I|=2}R(\widehat{K}_I)_k\stackrel{\partial_2}{\longrightarrow}\cdots\stackrel{\partial_n}{\longrightarrow}\bigoplus_{|I|=n+1}R(\widehat{K}_I)_k{\longrightarrow} 0\]
with the differential $\partial_p$ being the alternating sum of `twisted' restriction maps
\[\partial_p=\sum_{r=0}^p(-1)^r\rho_I^{I\cup \{i_r\}}\]
where $I=\{i_0, i_1, \cdots, i_{r-1}, i_{r+1}, \cdots, i_p\}$ with $0\leq i_0<i_1\cdots<i_{r-1}<i_r<i_{r+1}<\cdots<i_p\leq n$, and $\rho_{I_1}^{I_2}$ is the `twisted' restriction map $R(\widehat{K}_{I_1})_k\to R(\widehat{K}_{I_2})_k$ for $I_1\subset I_2$. The homology of the sequence vanishes except at the last slot, which is isomorphic to $K^n_K(K, (\tau_K^K)^{k+\textsf{h}^\vee})$ {(cf. \cite[Theorem 5.3]{M1})}. 

As mentioned at the beginning of this section, Segal's spectral sequence is equivalent to the generalized Mayer-Vietoris sequence applied to a {closed} cover specified by the simplicial structure: {If $p: K\to \Delta$ is the map taking a group element $k$ to the representative in $\Delta$ of the conjugacy class of $k$, and $V_i:=\Delta\setminus\mathcal{O}(\Delta_{1, 2, \cdots, \widehat{i}, \cdots, n})$ is the complement of an open neighbourhood of the subsimplex $\Delta_{1, 2, \cdots, \widehat{i}, \cdots, n}$ (the subsimplex opposite to the vertex $i$) in $\Delta$, then $\{p^{-1}(V_i)\}_{i=0}^n$ is a closed cover associated to the simplicial structure of $K$. It has the properties that $p^{-1}(V_i)$ is $K$-invariant, and $\bigcap_{i\in I}p^{-1}(V_i)$ deformation retracts onto any conjugacy class $\mathcal{C}$ with $p(\mathcal{C})\in\text{int}(\Delta_I)$, and the deformation retract is a proper map. The latter condition ensures that $K_K^*(\bigcap_{i\in I}p^{-1}(V_i))$ is isomorphic to $K_K^*(\mathcal{C})$ as $K$-theory is a compactly supported cohomology theory. This argument of using the Mayer-Vietoris sequence can be carried over to prove the Baum-Connes and Connes-Kasparov isomorphisms with coefficient $\Gamma_0(\tau_G^G)$, but with some modifications. As $G$ is not compact, there does not exist a finite closed cover such that there is an equivariant proper deformation retract of each closed set onto a conjugacy class of $G$. The following proposition gives the modified arguments which get around the issue caused by the non-compactness of $G$.}

\begin{proposition}\label{opencover}
	Let   {$G$ be a connected semi-simple Lie group with torsion-free fundamental group, and $K$ a maximal compact subgroup.} If there exists a finite {closed} cover $\{\mathcal{F}_i\}_{i=0}^n$ of $G$ such that
	{\begin{enumerate}
		\item $\mathcal{F}_i$ is invariant under the conjugation by $G$, and
		\item Any $\mathcal{F}_i$ and their intersection are stratifications of finitely many strata $\mathcal{O}_j$ (i.e. $Z_j:=\overline{\mathcal{O}}_j=\bigcup_{j\leq i}\mathcal{O}_j$ and $\mathcal{O}_j$ is open in $Z_j$) such that each stratum is either 
		\begin{enumerate}
			\item\label{type1} a finite disjoint union of fiber bundles over contractible sets $C$ whose one-point compactifications are contractible (e.g. closed cones) and on which $G$ acts trivially, or
			\item\label{type2}a finite disjoint union of conjugacy classes.
		\end{enumerate}
	\end{enumerate}}
	then
	\begin{enumerate}
		\item the assembly map 
		\[KK^G_{\bullet}(G/K, (G, \tau_G^G))\longrightarrow K_\bullet(\Gamma_0(\tau_G^G)\rtimes G)\]
		is an isomorphism.
		\item The Connes-Kasparov map
		\[K_\bullet(\Gamma_0(\tau_K^G)\rtimes K)\longrightarrow K_\bullet(\Gamma_0(\tau_G^G)\rtimes G)\]
		is an isomorphism.
	\end{enumerate}
\end{proposition}
\begin{proof}
	{We shall first prove the general fact that the Baum-Connes isomorphism holds for a stratification if it holds for each stratum. Note that there is a short exact sequence of $G$-$C^*$-algebras
	\[0\longrightarrow \Gamma_0(\tau_G^G|_{\mathcal{O}_j})\longrightarrow\Gamma_0(\tau_G^G|_{Z_j})\longrightarrow\Gamma_0(\tau_G^G|_{Z_{j-1}})\longrightarrow 0.\]
	 By \cite[Theorem  6.8]{EW}, we further have the short exact sequence of crossed product $C^*$-algebras
	\[0\longrightarrow\Gamma_0(\tau_G^G|_{\mathcal{O}_j})\rtimes G\longrightarrow\Gamma_0(\tau_G^G|_{Z_j})\rtimes G\longrightarrow\Gamma_0(\tau_G^G|_{Z_{j-1}})\rtimes G\longrightarrow 0.\]
	These two short exact sequences induce respectively two six-term exact sequences of $K$-theory groups connected by the Baum-Connes assembly maps:
	\begin{eqnarray*}
		\xymatrix{\ar[r]& KK_\bullet^G(G/K, (\mathcal{O}_i, \tau_G^G|_{\mathcal{O}_i}))\ar[r]\ar[d]^{\text{BC}}&KK_\bullet^G(G/K, (Z_i, \tau_G^G|_{Z_i}))\ar[r]\ar[d]^{\text{BC}}&KK_\bullet^G(G/K, (Z_{i-1}, \tau_G^G|_{Z_{i-1}}))\ar[r]\ar[d]^{\text{BC}}&\cdots\\ \ar[r]& K_\bullet(\Gamma_0(\tau_G^G|_{\mathcal{O}_i})\rtimes G)\ar[r]& K_\bullet(\Gamma_0(\tau_G^G|_{Z_i})\rtimes G)\ar[r]& K_\bullet(\Gamma_0(\tau_G^G|_{Z_{i-1}})\rtimes G)\ar[r]&\cdots}
	\end{eqnarray*}
	The above diagram commutes because the Baum-Connes assembly map is functorial with repsect to restriction and open inclusion. Using the above diagram and applying induction on $Z_i$,  {we prove} the Baum-Connes isomorphism for the stratification.}

{Now we would like to show that the Baum-Connes isomorphism holds for $\mathcal{F}_i$ and their intersection. To do this it suffices to show that the Baum-Connes isomorphism holds for the two types of strata in the proposition. For strata of type (\ref{type1}), i.e. a fiber bundle over $C$, it is necessarily trivial as $C$ is contractible. We may let this fiber bundle be $X\times C$, where $G$ acts on $C$ trivially. Let $C_+$ be the one-point compactification of $C$. Consider the short exact sequence of $G$-$C^*$-algebras
\[0 \longrightarrow \Gamma_0(\tau_G^G|_{X\times C}) \longrightarrow \Gamma_0(\tau_G^G|_{X\times C_+})\longrightarrow\Gamma_0(\tau_G^G|_{X\times \{+\}}) \longrightarrow 0.\]
Again by \cite[Theorem 6.8]{EW}, we further have the short exact sequence of crossed product $C^*$-algebras
\[0\longrightarrow\Gamma_0(\tau_G^G|_{X\times C})\rtimes G\longrightarrow\Gamma_0(\tau_G^G|_{X\times C_+})\rtimes G\longrightarrow\Gamma_0(\tau_G^G|_{X\times\{+\}})\rtimes G\longrightarrow 0.\]
By the fact that there is a proper equivariant deformation retract of $X\times C_+$ onto $X\times \{+\}$ and the two six-term exact sequences of $K$-theory groups induced by the above two short exact sequences, we have that both $KK_\bullet^G(G/K, (X\times C, \tau_G^G|_{X\times C}))$ and $K_\bullet(\Gamma_0(\tau_G^G|_{X\times C})\rtimes G)$ vanish and thus the Baum-Connes isomorphism holds for $X\times C$. }
	
	{Next we will show that the Baum-Connes assembly map is an isomorphism for strata of type (\ref{type2}), i.e. conjugacy classes. Let $H$ be the stabilizer subgroup of an element $g$ of the conjugacy class $\mathcal{C}$. So $\mathcal{C}\cong G/H$. Consider the following diagram
\begin{eqnarray}\label{conjugationinduction}
		\xymatrix{KK^G_\bullet(G/K, (\mathcal{C}, \tau_G^G|_{\mathcal{C}}))\ar[r]\ar[d]& K_\bullet(\Gamma_0(\tau_G^G|_{\mathcal{C}})\rtimes G)\ar[d]\\ KK^{H}_\bullet(G/K, (\{g\}, \tau_G^G|_{\{g\}}))\ar[r]& K_\bullet(\Gamma_0(\tau_G^G|_{\{g\}})\rtimes {H})}
	\end{eqnarray}
	The bottom map of Diagram (\ref{conjugationinduction}) is the Baum-Connes assembly map for $H$-action with coefficient $\mathcal{K}(\mathcal{H})$ (note that $G/K$ is also a universal space for proper $H$-action (\cite[Remark after Corollary 1.9]{BCH})) and an isomorphism by virtue of \cite[Theorem 1.2]{CEN}. It remains to {be shown} that Diagram (\ref{conjugationinduction}) is commutative and the two vertical maps are isomorphisms to establish that the top map of (\ref{conjugationinduction}) is indeed an isomorphism. The commutativity of the diagram is due to the fact that the element $[p]\in K_0(C_0(G/K)\rtimes G)$ used to define the Baum-Connes map for $G$-action as in the beginning of Section \ref{ringstr} restricts to the corresponding element $[q]\in K_0(C_0(G/K)\rtimes H)$ used to define the Baum-Connes map for $H$-action. The right vertical map is an isomorphism because by the slow-paced version of the path group construction of $\tau_G^G$ similar to that for $\tau_K^K$ as in \cite{M1}: $\tau_G^G|_{\mathcal{C}}$ is of the form $G\times_{H}\mathcal{K}(\mathcal{H}_Q)$ for some action by $H$ on $\mathcal{K}(\mathcal{H}_Q)$. It follows that $\text{Ind}_{H}^G(\Gamma_0(\tau_G^G|_{\{g\}})\rtimes H)=\Gamma_0(\tau_G^G|_{\mathcal{C}})\rtimes G$ and induction yields strongly Morita equivalent crossed product $C^*$-algebras (cf. \cite{P}). As to the left vertical map, one can establish that it is an isomorphism by adapting the proof of \cite[Lemma 2.4.2]{We} in a straightforward way, which again involves strong Morita equivalence arising from induction: we have the isomorphisms
	\begin{align*}
		&KK^G_\bullet(G/K, (\mathcal{C}, \tau_G^G|_{\mathcal{C}}))\\
		\cong &K_\bullet((\Gamma_0(\tau_G^G|_{\mathcal{C}})\otimes \Gamma_0(\text{Cliff}(TG/K)))\rtimes G)\\
		\stackrel{\text{Induction}}{\cong} &K_\bullet((\Gamma_0(\tau_G^G|_{g})\otimes \Gamma_0(\text{Cliff}(TG/K)))\rtimes H)\\
		\cong &KK_\bullet^{H}(G/K, (\{g\}, \tau_G^G|_{\{g\}})).
	\end{align*}
Thus the top map, which is the Baum-Connes assembly map for $\mathcal{C}$, is an isomorphism.} 
	
	{Finally, consider the two Mayer-Vietoris sequences with respect to the closed cover in the proposition for computing $KK^G_\bullet(G/K, (G, \tau_G^G))$ and $K_\bullet(\Gamma_0(\tau_G^G)\rtimes G)$ connected by various Baum-Connes maps. By the commutativity of the (twisted) restriction map and the Baum-Connes map (again because of the functoriality of the Baum-Connes map),  
the five-lemma and induction on the number of closed sets in the closed cover, we conclude that the assembly map for the coefficient $\Gamma_0(\tau_G^G)$ is an isomorphism.}

{Note that the same argument can be used to establish the Connes-Kasparov isomorphism with coefficient as well. }
\end{proof}
\begin{proposition}\label{goodcover}
	  {Any connected, real semi-simple \emph{linear} group with torsion-free fundamental group} satisfies the conditions in Proposition \ref{opencover}.
\end{proposition}
\begin{proof}
	For a real semisimple linear algebraic group $G$, consider the characters of its fundamental representations $\chi_1, \chi_2, \cdots, \chi_n$ (for a real semisimple Lie group which is not complex, we define its fundamental representations to be the restriction of those of the complexification of the Lie group). We have the Steinberg map
	\begin{align*}
		\text{St}:\ G&\to \mathbb{C}^n\\
		g&\mapsto (\chi_1(g), \chi_2(g), \cdots, \chi_n(g))
	\end{align*}
	which is conjugation invariant. A useful example to bear in mind is $G=SL(n+1, \mathbb{C})$, where the Steinberg map sends $g$ to the elementary symmetric polynomials {evaluated at the eigenvalues of $g$}. The key property of the Steinberg map which enables us to find a desired {closed} cover is that each fiber of the map is a {stratification} of conjugacy classes {as in Proposition \ref{opencover} (\ref{type2})}. 
	{The strata are indexed and partially ordered by the sizes of the Jordan blocks.} 
	Note that any fiber which is the most generic is itself one semisimple conjugacy class {which is regular}. In the case of $G=SL(n+1, \mathbb{C})$ such a fiber consists of matrices with the same list of distinct eigenvalues. In fact the diffeomorphism types of fibers of the Steinberg map (there are finitely many of them) can be partially ordered by genericity, and determined by their images under the Steinberg map. For example, when $G=SL(n+1, \mathbb{C})$ the diffeomorphism type of a fiber depends on the multiplicities of the eigenvalues of the matrices in that fiber, and a fiber with smaller multiplicities of eigenvalues is more generic. When $G=SL(n+1, \mathbb{R})$ the diffeomorphism type also depends on the number of conjugate pairs in the list of eigenvalues, in addition to multiplicities. In general, the image of the fibers of the same diffeomorphism type is a real semialgebraic set in $\mathbb{C}^n$ viewed as $\mathbb{R}^{2n}$, which by definition is a finite union of sets of the form
	\[\{(x_1, y_1, \cdots, x_n, y_n)\in\mathbb{R}^{2n}|\ f_i(x_1, y_1, \cdots, x_n, y_n)*0,\ f_i\text{ is a polynomial for }1\leq i\leq p\}\]
	where $*$ can be equality or inequalities (`$>$', `$<$', `$\leq$', `$\geq$' or `$\neq$'). Let $V_1, \cdots, V_m$ be the real semialgebraic sets, each of which is the image of the union of fibers of one diffeomorphism type. Then $\{V_i\}_i^m$ can be partially ordered by genericity, and gives a stratification of $\text{St}(G)$ in that
	\[\overline{V}_i=\bigcup_{V_j\leq V_i}V_j\]
	where $V_j\leq V_i$ means that $V_j$ is less generic than $V_i$.  {Note that $\text{St}^{-1}(V_i)\to V_i$ is a fiber bundle: each fiber $\text{St}^{-1}(v)$ is difeomorphic by definition, and a stratification of finitely many conjugacy classes, each of which has a canonical representative (e.g. Jordan canonical form for $SL(n, \mathbb{C})$ and real Jordan canonical form for $SL(n, \mathbb{R})$). These representatives give local sections of the map $\text{St}^{-1}(V_i)\to V_i$, showing that the map is locally trivial.}


{Now we  view $S^{2n}$ as the 1-point compactification of $\mathbb{R}^{2n}\cong\mathbb{C}^n$, denote the  {point at infinity} by $N$ and consider $\overline{V}_1, \cdots, \overline{V}_m$. By \cite[Proposition 2.5.9]{BCR}, those $\overline{V}_i$ which are not bounded admit one-point compactifications which are also real semi-algebraic and will be regarded as sitting inside $S^{2n}$. By applying \cite[Theorem 9.4.1]{BCR} which asserts finite triangulability of closed and bounded real semi-algebraic sets to those $\overline{V}_i$ which are bounded and the one-point compactifications of those $\overline{V}_j$ which are not bounded, we have that there exists a finite simplicial complex $T$ and a homeomorphism
	\[\Phi: T\to S^{2n}\]
	such that each set $\overline{V}_i$ is the image under $\Phi$ a union of open subsimplices of $T$, and $N=\Phi(v')$ for some vertex $v'$ of $T$. By $V_i=\overline{V}_i\setminus\bigcup_{V_j<V_i}V_j$ and induction on the strata $V_i$, we have that each $V_i$ is also the image of under $\Phi$ of a union of open subsimplices.}
	
{Let $\{v_1, \cdots, v_m\}$ be the set of vertices of $T$ excluding $v'$, $\widetilde{T}$ be the simplicial complex obtained by the barycentric subdivision of $T$, and $\ell_1, \cdots, \ell_p$ be the 1-dimensional open subsimplices (i.e. open intervals) in $\widetilde{T}$ whose closure contains $v'$. We define $\text{Star}_{\widetilde{T}}(D)$ (respectively $\text{Star}^k_{\widetilde{T}}(D)$), the star (respectively the $k$ dimensional star) in $\widetilde{T}$ of an open subsimplex $D$, to be the union of all the open subsimplices (respectively the $k$ dimensional open subsimplices) whose closure contains the given open subsimplex. Let 
\begin{align*}
	\mathcal{F}_i&:=\text{St}^{-1}(\Phi(\overline{\text{Star}_{\widetilde{T}}(v_i)})),\ 1\leq i\leq m\\
	\mathcal{F}'_j&:=\text{St}^{-1}(\Phi(\overline{\text{Star}_{\widetilde{T}}(\ell_j)}\setminus\{N\})), 1\leq j\leq p.
\end{align*} 
Then we have the following.
\begin{enumerate}
	\item\label{starvertex} Let $d$ be the dimension of $\overline{\text{Star}_{\widetilde{T}}(v_i)}$. Then $\overline{\text{Star}_{\widetilde{T}}(v_i)}$ is a stratification of $\mathcal{O}_0:=\text{Star}_{\widetilde{T}}^0(v_i)=v_i$, $\mathcal{O}_1:=\overline{\text{Star}_{\widetilde{T}}^1(v_i)}\setminus\overline{\text{Star}_{\widetilde{T}}^0(v_i)}$, $\cdots, \mathcal{O}_{k+1}:=\overline{\text{Star}_{\widetilde{T}}^{k+1}(v_i)}\setminus\overline{\text{Star}_{\widetilde{T}}^{k}(v_i)}, \cdots$, $\mathcal{O}_d:=\overline{\text{Star}_{\widetilde{T}}^d(v_i)}\setminus\overline{\text{Star}_{\widetilde{T}}^{d-1}(v_i)}$. Each stratum, except $\mathcal{O}_0$, is a finite disjoint union of half-open subsimplices $C$ of $\widetilde{T}$ whose one-point compactifications are contractible. Each of these half-open subsimplices $C$ is a subset of an open subsimplex of $T$ and hence a subset of some $V_i$. It follows that each of $\text{St}^{-1}(\Phi(\mathcal{O}_k))$ for $1\leq k\leq d$ is a finite disjoint union of fiber bundles over half-open subsimplices $C$ of $\widetilde{T}$. Moreover, $\text{St}^{-1}(\Phi(\mathcal{O}_0))$ is a stratification of conjugacy classes, as we noted before. Hence $\mathcal{F}_i$ is a stratification of sets of types (\ref{type1}) and (\ref{type2}) in Proposition \ref{opencover}.
	\item\label{starline} Similarly, $\overline{\text{Star}_{\widetilde{T}}(\ell_j)}\setminus\{v'\}$ is stratified by $\mathcal{O}_1:=\overline{\text{Star}_{\widetilde{T}}^1(\ell_j)}\setminus\{v'\}=\overline{\ell_j}\setminus\{v'\}$, $\mathcal{O}_2:=\overline{\text{Star}_{\widetilde{T}}^2(\ell_j)}\setminus\overline{\text{Star}_{\widetilde{T}}^1(\ell_j)}, \cdots$, $\mathcal{O}_k:=\overline{\text{Star}_{\widetilde{T}}^k(\ell_j)}\setminus\overline{\text{Star}_{\widetilde{T}}^{k-1}(\ell_j)}$, $\cdots, \mathcal{O}_d:=\overline{\text{Star}_{\widetilde{T}}^d(\ell_j)}\setminus\overline{\text{Star}_{\widetilde{T}}^{d-1}(\ell_j)}$, where $d$ is the dimension of $\overline{\text{Star}_{\widetilde{T}}(\ell_j)}$. Each stratum again is a finite disjoint union of half-open subsimplices of $\widetilde{T}$, each of which is a subset of some $V_i$ and has contractible one-point compactification. Thus the inverse images $\text{St}^{-1}(\Phi(\mathcal{O}_1)), \cdots, \text{St}^{-1}(\Phi(\mathcal{O}_d))$ are of type (\ref{type1}) in Proposition \ref{opencover} and stratify $\mathcal{F}'_j$.
	\item The intersection $\bigcap_{i\in I}\overline{\text{Star}_{\widetilde{T}}(v_i)}\cap\bigcap_{j\in J}(\overline{\text{Star}_{\widetilde{T}}(\ell_j)}\setminus\{v'\})$, where the set $I$ is non-empty and $J$ may be empty, is the closure of the union of some open subsimplices of $\widetilde{T}$ whose closure contains $b$, the barycenter of the closed subsimplex of the least dimension of $T$ containing $\{v_i\}_{i\in I}\cup\{\ell_j\}_{j\in J}$. Similar to Case (\ref{starvertex}), this intersection is stratified by $\{b\}$ and a finite disjoint union of half-open subsimplices of $\widetilde{T}$ each of which is a subset of some $V_i$ and has contractible one-point compactification. Hence $\bigcap_{i\in I}\mathcal{F}_i\cap\bigcap_{j\in J}\mathcal{F}'_j$ is a stratification of sets of types (\ref{type1}) and (\ref{type2}) in Proposition \ref{opencover}. 
	\item The intersection $\bigcap_{j\in J}(\overline{\text{Star}_{\widetilde{T}}(\ell_j)}\setminus \{v'\})$ is $A\setminus\{v'\}$, where $A$ is the closure of the union of some open subsimplices of $\widetilde{T}$ whose closure contains $\ell_{bv'}$, where $b$ is the barycenter of the closed subsimplex of the least dimension of $T$ containing $\{\ell_j\}_{j\in J}$ and $\ell_{bv'}$ is the open 1-dimensional subsimplex with endpoints $b$ and $v'$. Similar to Case (\ref{starline}), this intersection is stratified by a finite disjoint union of half-open subsimplices of $\widetilde{T}$ each of which is a subset of some $V_i$ and has contractible one-point compactification. Hence $\bigcap_{j\in J}\mathcal{F}'_j$ is a stratification of sets of type (\ref{type1}) in Proposition \ref{opencover}.
\end{enumerate}
We conclude that $\{\mathcal{F}_i\}_{i=1}^m\cup\{\mathcal{F}'_j\}_{j=1}^p$ is a closed cover of $G$ satisfying the conditions in Proposition \ref{opencover}.}
\end{proof}
\begin{theorem}\label{mainthm}
	\begin{enumerate}
		\item If   {$G$ is a connected, real semisimple linear group with torsion-free fundamental group}, then both the Baum-Connes map
	\[KK^G_\bullet(G/K, (G, \tau_G^G))\to K_\bullet(\Gamma_0(\tau_G^G)\rtimes G)\]
	and the Connes-Kasparov map
	\[K_\bullet(\Gamma_0(\tau_K^G)\rtimes K)\to K_\bullet(\Gamma_0(\tau_G^G)\rtimes G)\]
	are isomorphism.
		\item\label{ringstructure} The equivariant twisted $K$-homology ring $K^K_\bullet(K, \tau_K^K)$ is naturally isomorphic to both $KK^G_\bullet(G/K, (G, \tau_G^G))$ and $K_\bullet(\Gamma_0(\tau_K^G)\rtimes K)$ and hence induces ring structures on \\ $KK^G_\bullet(G/K, (G, \tau_G^G))$, $K_\bullet(\Gamma_0(\tau_K^G)\rtimes K)$, as well as $K_\bullet(\Gamma_0(\tau_G^G)\rtimes G)$ through the Baum-Connes map (or Connes-Kasparov map).
	\end{enumerate}
\end{theorem}
\begin{proof}
	\begin{enumerate}
		\item This is a consequence of Propositions \ref{opencover} and \ref{goodcover}.
		\item The Poincar\'e duality asserts that 
		 {\[\text{PD}: K^K_\bullet(K, \tau_K^K)\cong K_K^{\bullet{+\text{dim }K}}(K, \tau_K^K).\]}
		{As explained in the end of Section \ref{thetwists}, }
		 {\[i^*: K_K^\bullet(G, \tau_K^G)\to K_K^{\bullet{+\text{dim }K-\text{dim }G}}(K, \tau_K^K)\]}
		{is an isomorphism}. By the Green-Julg Theorem, we have that
		\[K_K^\bullet(G, \tau_K^G)\cong K_\bullet(\Gamma_0(\tau_K^G)\rtimes K).\]
		On the other hand, $K_K^\bullet(G, \tau_K^G)$ is by definition the KK-group $KK^K_\bullet(\text{pt}, (G, \tau_K^G))$, and we have the {equivariant Thom isomorphism}
		 \[\begin{aligned}
		  {[G/K]\otimes\cdot}: KK^K_\bullet(\text{pt}, (G, \tau_K^G))\to KK^K_{\bullet+\text{dim }G-\text{dim }K}(G/K, (G, \tau_K^G))\\ (\text{Since}\, G/K\text{ is }K\text{-equivariantly diffeomorphic to }\mathfrak{p}).
		 \end{aligned}\]
		By \cite[Theorem 5.8]{Ka} (phrased as \cite[Proposition 3.2]{MN}), which asserts the isomorphism of the restriction map, we finally get
		 {\[(\text{Res}_G^K)^{-1}: KK^K_\bullet(G/K, (G, \tau_K^G))\cong KK_\bullet^G(G/K, (G, \tau_G^G)).\] }
	\end{enumerate}
\end{proof}
\begin{example}\label{SL(2, C)}
	We exhibit {two} finite {closed covers} of $SL(2, \mathbb{C})$ satisfying the conditions in Proposition \ref{opencover} and hence show that the Baum-Connes and Connes-Kasparov isomorphisms with coefficient are true for $SL(2, \mathbb{C})$. Consider the {Steinberg map of $SL(2, \mathbb{C})$, which is the }trace map
	\[\text{Tr}: SL(2, \mathbb{C})\to \mathbb{C}.\]
	If $r\neq \pm 2$, $\text{Tr}^{-1}(r)$ is a conjugacy class in $SL(2, \mathbb{C})$ diffeomorphic to $SL(2, \mathbb{C})/A$ where $A$ is the subgroup of diagonal matrices. If $r=2$ or $-2$, then $\text{Tr}^{-1}(r)$ consists of two conjugacy classes: the identity $I_2$ (respectively $-I_2$) and the unipotent conjugacy class (respectively the negative of the unipotent conjugacy class). Note that 
	$$\text{Tr}^{-1}(\pm 2)=\left\{\left.\begin{pmatrix}a& b\\ c& \pm2-a\end{pmatrix}\right| a(\pm 2-a)-bc=1,\ a,\ b,\ c\in\mathbb{C}\right\}$$ is a complex affine quadric surface $V((a\pm 1)^2+bc)\subset\mathbb{C}^3$ which is singular at the point $(a, b, c)=(\pm 1, 0, 0)$ representing $\pm I_2$. {So $\text{Tr}^{-1}(\pm 2)$ indeed are stratifications of conjugacy classes as the closure of the unipotent conjugacy classes contains $\pm I_2$. We let $V_1=\{2, -2\}$, the image under the trace map of the less generic conjugacy classes, and $V_2=\mathbb{C}\setminus V_1$, the image of the more generic conjugacy classes. We can take the simplicial complex $T$ as in the proof of Proposition \ref{goodcover} to be a tetrahedron where $\pm 2$ and the point at infinity $N$ are three of the vertices and another complex number (say $i$) is the remaining vertex. One can get a desired closed cover of $SL(2, \mathbb{C})$ by following the construction of the barycentric subdivision of $T$ in the proof of Proposition \ref{goodcover}. Alternatively, there is a more economical closed cover, which consists of 
	\begin{enumerate}
		\item\label{square} the inverse images under the trace map of the two closed squares of side length 2 centered at $2$ and $-2$ respectively, and 
		\item\label{quadrant} the inverse images of the closure of the complement of the two closed squares in the four closed quadrants.
	\end{enumerate}
	Note that each of the inverse images in (\ref{square}) is a stratification of a fiber bundle with fibers being the more generic conjugacy classes, over the complement of a singleton in a closed square (whose one-point compactification is contractible), and the inverse image of that singleton under the trace map which itself is a stratification of the unipotent conjugacy class and the conjugacy class consisting of one element ($I_2$ or $-I_2$). Each of the inverse images in (\ref{quadrant}) is a fiber bundle with fibers being the more generic conjugacy classes, over a set homeomorphic to a closed quadrant whose one-point compactification is also contractible. Moreover, the intersection of these inverse images also satisfies the conditions in Proposition \ref{opencover}.} 
\end{example}
\begin{example}\label{SL(2, R)}
	For $G=SL(2, \mathbb{R})$, there is a similar analysis of its conjugacy classes using the trace map and an analogous construction of a {closed} cover satisfying the conditions in Proposition \ref{opencover}, though the arrangement of the conjugacy classes is more complicated than that of $SL(2, \mathbb{C})$ due to the fact that the ground field $\mathbb{R}$ is not algebraically closed. When $r>2$ or $r<-2$, $\text{Tr}^{-1}(r)$ is a conjugacy class diffeomorphic to $SL(2, \mathbb{R})/A$, where $A$ is the subgroup of diagonal matrices. When $-2<r<2$, $\text{Tr}^{-1}(r)$ is a conjugacy class diffeomorphic to $SL(2, \mathbb{R})/K$, where $K=SO(2)$. For $r=\pm 2$, $\text{Tr}^{-1}(r)$ consists of three conjugacy classes: $I_2$ (respectively $-I_2$), $\left\{\left.g\begin{pmatrix}1&b\\ 0&1\end{pmatrix}g^{-1}\right|b>0, g\in SL(2, \mathbb{R})\right\}$ (respectively $\left\{\left.g\begin{pmatrix}-1&b\\ 0&-1\end{pmatrix}g^{-1}\right|b>0, g\in SL(2, \mathbb{R})\right\}$) and $\left\{\left.g\begin{pmatrix}1&b\\ 0&1\end{pmatrix}g^{-1}\right|b<0, g\in SL(2, \mathbb{R})\right\}$ (respectively $\left\{\left.g\begin{pmatrix}-1&b\\ 0&-1\end{pmatrix}g^{-1}\right|b<0, g\in SL(2, \mathbb{R})\right\}$). $\text{Tr}^{-1}(r)$ in fact is isomorphic to the real cone $V(x^2+y^2-z^2)\subset\mathbb{R}^3$ where the singular point corresponds to $\pm I_2$ and the two connected components of the complement of the singular point corresponds to the two other conjugacy classes. {So $\text{Tr}^{-1}(r)$ is a stratification of conjugacy classes. We let $V_1=\{-2, 2\}$, $V_2=(-\infty, -2)\cup(2, \infty)$ and $V_3=(-2, 2)$, and note by the above analysis that $V_1$ is the image under the trace map of the less generic conjugacy classes while $V_2$ and $V_3$ are the images of the more generic conjugacy classes. The simplicial complex $T$ can be taken to be a triangle where $\pm 2$ and the point at infinity $N$ are the three vertices. Following again the construction in the proof of Proposition \ref{goodcover}, we take 
	\begin{align*}
		\mathcal{F}_1&:=\text{Tr}^{-1}((-\infty, -3]),\\
		\mathcal{F}_2&:=\text{Tr}^{-1}([-3, 0]),\\
		\mathcal{F}_3&:=\text{Tr}^{-1}([0, 3]),\\
		\mathcal{F}_4&:=\text{Tr}^{-1}([3, \infty)),
	\end{align*}
	where the three points $\pm 3$ and 0 act as the three barycenters of the three sides of the triangle.}
\end{example}
\begin{remark}
	If $G_1$ and $G_2$ are two real semisimple Lie groups with {closed} covers satisfying the conditions in Proposition \ref{opencover}, then the Cartesian product of the two {closed} covers is a {closed} cover for $G:=G_1\times G_2$ also satisfying the same conditions. It follows that any product of $SL(2, \mathbb{C})$ and $SL(2, \mathbb{R})$ witnesses the Baum-Connes and Connes-Kasparov isomorphisms with the coefficient of the compact algebra bundle.
\end{remark}

\subsection{Pontryagin product of $K^K_\text{geom}(K, \tau_K^K)$ in terms of geometric cycles}
 {Let $K$ be a simple and simply-connected compact Lie group. Let $\mu\in\Lambda^*_{k-\textsf{h}^\vee}(K)$ be a level $k-\textsf{h}^\vee$ weight and $\mathcal{C}_\mu$ the conjugacy class in $K$ containing $\text{exp}\left(\frac{B^\sharp(\mu)}{k-\textsf{h}^\vee}\right)$. In \cite{M1}, it is shown that the pushforward map 
\[K^K_\bullet(\mathcal{C}_\mu, \text{Cl}(T\mathcal{C}_\mu))\to K^K_\bullet(K, \tau_K^K)\]
induced by the inclusion $\iota_\mu$ of $\mathcal{C}_\mu$ into $K$ sends the fundamental class of $\mathcal{C}_\mu$ to the generator $\sigma_\mu$ (the isomorphism class of the irreducible positive energy representation of $LK$ with lowest weight $\mu$) of $R^{k-\textsf{h}^\vee}(LK)$, which is identified with $K^K_\bullet(K, \tau_K^K)$ by the Freed-Hopkins-Teleman Theorem. It follows that the geometric $K$-homology $K^K_\text{geom}(K, \tau_K^K)$ is generated by the geometric cycles
\[(\mathcal{C}_\mu, \iota_\mu, T\mathcal{C}_\mu), \mu\in\Lambda_{k-\textsf{h}^\vee}^*(K).\]
We will abbreviate the above cycle as $[\mathcal{C}_\mu]$. Recall that the Pontryagin product of $K_\text{geom}^K(K, \tau_K^K)$, which corresponds to the fusion product of the Verlinde algebra through the Freed-Hopkins-Teleman Theorem, is induced by the group multiplication of $K$. In terms of geometric cycles, the Pontryagin product is given by 
\[[\mathcal{C}_{\mu_1}]\cdot[\mathcal{C}_{\mu_2}]=[\mathcal{C}_{\mu_1}\cdot\mathcal{C}_{\mu_2}]\]
where $\mathcal{C}_{\mu_1}\cdot\mathcal{C}_{\mu_2}=\{k_1k_2\in K| k_i\in\mathcal{C}_i\}$. Let $c_{\mu_1\mu_2}^{\mu_3}$ be the structure constant of $R^k(LK)$ defined by 
\[\sigma_{\mu_1}\cdot\sigma_{\mu_2}=\sum_{\mu_3\in\Lambda^*_k(K)}c_{\mu_1\mu_2}^{\mu_3}\sigma_{\mu_3}.\]
It is well-known that the structure constant is nonnegative. By the Freed-Hopkins-Teleman Theorem, we have that 
\[[\mathcal{C}_{\mu_1}\cdot\mathcal{C}_{\mu_2}]=[\coprod_{\mu_3\in\Lambda_k^*(K)}c_{\mu_1\mu_2}^{\mu_3}\mathcal{C}_{\mu_3}].\]}

The multiplicative Horn's Problem, which asks for possible eigenvalues of product of two unitary matrices with given eigenvalues, is known to admit a solution involving the quantum cohomology of Grassmannians (\cite{Bel}), which in turn is isomorphic to the Verlinde algebra when the quantum variable is specialized to 1 (\cite{Wi}). The above interpretation of the Pontryagin product in terms of geometric cycles and the Freed-Hopkins-Teleman Theorem gives an alternative explanation why the multiplicative Horn's Problem involves the Verlinde algebra.

\section{Generators for $KK^G_\bullet(G/K, (G, \tau_G^G))$}\label{sect:generators}

Recall that for  {$K$ a connected and simply-connected compact Lie group}, then its conjugacy classes are in bijection with the quotient $T/W$ of its maximal torus $T$ by the action of the Weyl group $W$.
Meinrenken \cite{M1}
has shown that in this case, the generators of $K_*^{K, geom}(K, \tau_K^K)$ are  { twisted $K$-homology geometric cycles $({\mathcal C}_\mu, T{\mathcal C}_\mu, \iota)$}, where $\mu\in \Lambda_{k-\textsf{h}^\vee}^*(K)$ and ${\mathcal C}_\mu$ is a conjugacy class in $K$  {containing $\text{exp}\left(\frac{B^\flat(\mu)}{k-\textsf{h}^\vee}\right)$}, 
$T{\mathcal C}_\mu$ is the tangent bundle to ${\mathcal C}_\mu$, $\iota: {\mathcal C}_\mu \hookrightarrow K$ is the inclusion map.

Let  {$G$ be a connected, simply-connected Lie group}.  {In this section we give a list of generators of $KK^G_\bullet(G/K, (G, \tau_G^G))$ based on those of $K^K_\bullet(K, \tau_K^K)$ using \emph{correspondences} of twisted equivariant $KK$-theory to be described in \S \ref{sect:correspondences}.} Define ${\mathcal D}_\mu = G\times_K {\mathcal C}_\mu$, on which $G$ acts by left translation on the first factor and trivially on the second factor. Then there is a natural projection map 
$b: {\mathcal D}_\mu \longrightarrow G/K$ {which is proper as the fibers of the map are $\mathcal{C}_\mu$ which is compact}. There is also a natural twisted $G$-${\text{spin}^c}$ map $f: {\mathcal D}_\mu \longrightarrow G$ defined by,  {$[g, c] \longrightarrow gcg^{-1}$}. Then we can form the correspondences,

\begin{eqnarray}\label{generatorcorrespondence}
\xymatrix{& \ar[ddl]_b ({\mathcal D}_\mu , \xi)\ar[ddr]^{f} &\\&&\\G/K && G}
\end{eqnarray}
where $\xi_\mu = G\times_K T{\mathcal C}_\mu$. 
\begin{proposition}
	The correspondence (\ref{generatorcorrespondence}) for $\mu\in\Lambda_{k-\textsf{h}^\vee}(K)$ are generators of $KK^G_\text{geom}(G/K, (G, \tau_G^G))$.
\end{proposition}
\begin{proof}
	We shall start from the class represented by the correspondence
	\begin{eqnarray*}
		\xymatrix{&(\mathcal{C}_\mu, T\mathcal{C}_\mu) \ar[ddl]_{b_1} \ar[ddr]^{f_1}&\\
		&&\\
		\text{pt} &&K}
	\end{eqnarray*}
	in $KK^K_\text{geom}(\text{pt}, (K, \tau_K^K))$ where $f_1$ is the inclusion map and by tracing the various isomorphisms in the proof of Theorem \ref{mainthm} (\ref{ringstructure}) show that the associated class in $KK^G_\text{geom}(G/K, (G, \tau_G^G))$ can be represented by the correspondence (\ref{generatorcorrespondence}). The isomorphism
	\[ {[G/K]\otimes\cdot: KK^K_\bullet(\text{pt}, (K, \tau_K^K))\cong KK^K_\bullet(G/K, (K, \tau_K^K))}\]
	 {takes the above correspondence to the intersection product of it and $[G/K]$, which is represented by the correspondence}
	\begin{eqnarray*}
		 {\xymatrix{& (\mathfrak{p}, \textbf{1}) \ar[ddl]_{b_2} \ar[ddr]^{f_2}&\\ &&\\ G/K &&\text{pt}}}
	\end{eqnarray*}
	 {where $b_2$ is the diffeomorphism}
	\begin{align*}
		\mathfrak{p}&\to G/K\\ 
		\zeta&\mapsto \text{exp}(\zeta)K,
	\end{align*}
	 {$f_2$ the collapsing map and $\textbf{1}$ the trivial complex line bundle over $\mathfrak{p}$. The intersection product is the following correspondence} 
	\begin{eqnarray}\label{inflatedKcorrespondence}
		 {\xymatrix{&(\mathfrak{p}\times\mathcal{C}_\mu, \mathfrak{p}\times T\mathcal{C}_\mu)\ar[ddl]_{b_3} \ar[ddr]^{f_3}&\\ &&\\ G/K&& K}}
	\end{eqnarray}
	 {where} 
	\begin{align*}
		b_3: \mathfrak{p}\times\mathcal{C}_\mu&\to G/K\\
		(\zeta, k)&\mapsto \text{exp}(\zeta)K,\ \text{and}\\
		f_3: \mathfrak{p}\times \mathcal{C}_\mu&\to K\\
		(\zeta, k)&\mapsto k
	\end{align*}
	{The map $b_3$ is proper because its fibers are diffeomorphic to $\mathcal{C}_\mu$, which is compact. The Thom isomorphism }
	\[ {KK^K_{\bullet}(G/K, (K, \tau_K^K))\to KK^K_\bullet(G/K, (G, \tau_K^G))}\]
	 {takes the correspondence (\ref{inflatedKcorrespondence}) to } 
	
	\begin{eqnarray}
		\xymatrix{&(\mathfrak{p}\times\mathcal{C}_\mu, \mathfrak{p}\times T\mathcal{C}_\mu)\ar[ddl]_{b_3} \ar[ddr]^{f_4}&\\ &&\\ G/K&&  {G}}
	\end{eqnarray}
	where $f_4=i\circ f_3.$
	The $K$-equivariant diffeomorphism 
	\begin{align*}
		\mathfrak{p}\times\mathcal{C}_\mu&\to G\times_K\mathcal{C}_\mu\\
		(\zeta, k)&\mapsto [\text{exp}(\zeta)k, k]
	\end{align*}
	takes the correspondence (\ref{inflatedKcorrespondence}) to the correspondence
	\begin{eqnarray}\label{7thcorrespondence}
		\xymatrix{&(\mathcal{D}_\mu, \xi_\mu) \ar[ddl]_{b} \ar[ddr]^{ {f_5}}&\\ && \\ G/K&& G}
	\end{eqnarray}
	where, for the Cartan decomposition of any group element $g=\text{exp}(\zeta)k_0$, 
	\begin{align*}
		b: \mathcal{D}_\mu&\to G/K\\
		[g, k]&\mapsto gK,\ \text{and}\\
		 {f_5}: \mathcal{D}_\mu&\to G\\
		[g, k]&\mapsto gkk_0^{-1}=\text{exp}(\zeta)k_0kk_0^{-1}.
	\end{align*}
	The map $ {f_5}$ is $K$-equivariantly homotopy equivalent to the map 
	\begin{align*}
		f: \mathcal{D}_\mu&\to G\\
		[g, k]&\mapsto gkg^{-1}
	\end{align*}
	through the homotopy $([g, k], t)\mapsto \text{exp}(\zeta)k_0kk_0^{-1}\text{exp}(-t\zeta)$. Thus the correspondences (\ref{7thcorrespondence}) and (\ref{generatorcorrespondence}) represent the same $KK$-class in $KK^K_\text{geom}(G/K, (G, \tau_K^G))$. By \cite[Theorem 5.8]{Ka} (phrased as \cite[Proposition 3.2]{MN}), we have that the restriction map 
	\[KK^G_\bullet(G/K, (G, \tau_G^G))\to KK^K_\bullet(G/K, (G, \tau_K^G))\]
	 is an isomorphism. The correspondence (\ref{generatorcorrespondence}) with $K$-action representing a $KK$-class in the latter $KK$-group then is associated under this isomorphism to the same correspondence but equipped with $G$-action. All in all, we have shown that the geometric cycle $(\mathcal{C}_\mu, T\mathcal{C}_\mu, \iota)$ is associated to the correspondence (\ref{generatorcorrespondence}) under the isomorphism $K^K_\bullet(K, \tau_K^K)\cong KK^G_\bullet(G/K, (G, \tau_K^G))$ which is the composition of various isomorphisms in the proof of Theorem \ref{mainthm} (\ref{ringstructure}). Noting that $(\mathcal{C}_\mu, T\mathcal{C}_\mu, \iota)$ for $\mu\in \Lambda_{k-\textsf{h}^\vee}^*(K)$ represent a set of generators of $K^K_\text{geom}(K, \tau_K^K)$, we complete the proof. 
\end{proof}
We next describe the ring structure of $KK^G_\text{geom}(G/K, (G, \tau_G^G))$  in terms of the generators as described previously. Recall that from \cite{FHT} that the product
$$
[{\mathcal C}_{\mu_1}, T{\mathcal C}_{\mu_1}, \iota_1]\star [{\mathcal C}_{\mu_2}, T{\mathcal C}_{\mu_2}, \iota_2] = \sum_\mu c^\mu_{\mu_1, \mu_2} 
[{\mathcal C}_{\mu}, T{\mathcal C}_{\mu}, \iota] 
$$
where $c^\mu_{\mu_1, \mu_2}$ are the structure constants of the Verlinde algebra with respect to the generators. Therefore in the noncompact case, the product of correspondences 
yielding the ring structure in $KK^G_{geom}(G/K, (G, \tau_G^G))$ is,
 \begin{equation}
{\tiny \xymatrix{& \ar[ddl]_b ({\mathcal D}_{\mu_1} , \xi_{\mu_1})\ar[ddr]^{f} &\\
&&\\
G/K && G} \star
\xymatrix{& \ar[ddl]_b ({\mathcal D}_{\mu_2} , \xi_{\mu_2})\ar[ddr]^{f} &\\
&&\\
G/K && G} = \sum_\mu c^\mu_{\mu_1, \mu_2} 
\xymatrix{& \ar[ddl]_b ({\mathcal D}_{\mu} , \xi_{\mu})\ar[ddr]^{f} &\\
&&\\
G/K && G}}
\end{equation}

\section{The geometric description of $KK^G(\Gamma_0(\text{Cliff}(TX)), \Gamma_0( \tau_Y))$}\label{sect:correspondences}

In this section, we will give a geometric description of $KK^G_\bullet(X, (Y, \tau_Y))$ in terms of  certain correspondences. The description of 
$KK(X,Y)$ in terms of correspondences is due to  {Connes-Skandalis \cite{ConnesSkandalis84}}, which we adapt to the twisted equivariant case, clarifying and adapting \cite{EM} in our context. 
We also mention that the geometric model for twisted $K$-homology is due to Wang \cite{Wang}.  In this section we will assume that $G$ is a connected Lie group. The reason that we do this is because under this assumption, it has been proved by Phillips \cite{Phillips2} that the $G$-equivariant $K$-theory, $K_G^0(M)$ is generated by finite dimensional $G$-vector bundles, whenever $G$ acts properly cocompactly on $M$.  Examples by Phillips in \cite{Phillips1} and L\"uck-Oliver in \cite{LO} \S5, show that for an {\em arbitrary} Lie group $G$, the equivariant $K$-theory constructed from finite-dimensional $G$-vector bundles may not be a generalised cohomology theory.

As before, we assume that the $G$ action on manifolds in this section is  proper.

\subsection{Geometric model, $KK^G_{geom}(\Gamma_0(\text{Cliff}(TX)), \Gamma_0( \tau_Y))$}
A twisted $G$-equivariant correspondence $\mathcal C$ between the  $G$-spaces $X$ and $(Y, \tau_Y)$, where $\tau_Y$ is a $G$-DD-bundle over $Y$, is a diagram
 \begin{equation}
\xymatrix{& \ar[ddl]_b (M , \xi)\ar[ddr]^{f} &\\
&&\\
X && Y}
 \end{equation}
where  $M$ is a $G$-space, $\xi$  is a $G$-vector bundle on $M$ that is also a module for the $G$-algebra bundle $\text{Cliff}(TM)$. Here $b$ is a proper $G$-map, $f$ is smooth $G$-map 
that is twisted ${\text{spin}^c}$, i.e.,
\begin{equation}
f^*(\tau_Y) {\otimes} \text{Cliff}(TM) \cong {\rm End}(S)
\end{equation}
for some vector bundle $S\to M$ over $M$.

A pair of twisted $G$-equivariant correspondences
 \begin{equation}
\xymatrix{& \ar[ddl]_b (M , \xi)\ar[ddr]^{f} &\\
&&\\
X && Y} \text{and}
\xymatrix{& \ar[ddl]_{b_1} (M_1 , \xi_1)\ar[ddr]^{f_1} &\\
&&\\
X && Y}
 \end{equation}
 are said to be {\em isomorphic}, denoted $\cong$, if there is a $G$-diffeomorphism $h:M\to M_1$ such that 
 $\xi \cong h^*(\xi_1)$ as $G$-Clifford module bundles over $M$.\\

\begin{definition}
Denote by $KK^G_{geom}(\Gamma_0(\text{Cliff}(TX)), \Gamma_0( \tau_Y))$ the set of equivalence classes of twisted $G$-equivariant correspondences over $(X, Y)$, with the equivalence relation generated by the following
below.

\begin{enumerate}
\item Direct sum - disjoint union:
 \begin{equation} \label{directsum}
\xymatrix{& \ar[ddl]_b (M, \xi_1)\ar[ddr]^{f} &\\
&&\\
X && Y} \quad \bigsqcup
\xymatrix{& \ar[ddl]_b (M, \xi_2)\ar[ddr]^{f} &\\
&&\\
X && Y} \cong
\xymatrix{& \ar[ddl]_b (M_1 , \xi_1)\ar[ddr]^{f} &\\
&&\\
X && Y}\end{equation}

\item Equivariant bordism: Two twisted $G$-correspondences,

 \begin{equation}
\xymatrix{& \ar[ddl]_{b_1} (M_1 , \xi_1)\ar[ddr]^{f_1} &\\
&&\\
X && Y}
\quad\text{and}\quad
\xymatrix{& \ar[ddl]_{b_2} (M_2 , \xi_2)\ar[ddr]^{f_2} &\\
&&\\
X && Y}
 \end{equation}
are said to be {\em $G$-bordant} if there is a twisted $G$-correspondence with boundary,
 \begin{equation}
\xymatrix{& \ar[ddl]_b (M , \xi)\ar[ddr]^{f} &\\
&&\\
X && Y}
 \end{equation}
such that $M$ is a $G$-manifold with boundary $\partial M= M_1 \bigsqcup M_2$, $\xi\big|_{M_k}=\xi_k$, 
$f\big|_{M_k}=f_k$ and $g\big|_{M_k}=(-1)^{k+1}g_k$ for $k=1,2$.\\

\item Equivariant vector bundle modification:
Let $W$ be a $G$-${\text{spin}^c}$ vector bundle over $M$ with even dimensional fibres and which is also a module over $\text{Cliff}(TM)$. 
Let $Z=S(W\oplus {\bf 1})$ be the unit sphere bundle of the odd dimensional vector bundle $W\oplus {\bf 1})$, and let 
$\pi: Z\longrightarrow M$ denote the projection. Then the correspondence obtained from
 \begin{equation}
\xymatrix{& \ar[ddl]_b (M , \xi)\ar[ddr]^{f} &\\
&&\\
X && Y}
 \end{equation}
by vector bundle modification is,
 \begin{equation}
\xymatrix{& \ar[ddl]_{b\circ\pi} (Z , F \otimes \pi^*\xi)\ar[ddr]^{f\circ\pi} &\\
&&\\
X && Y}
 \end{equation}
Here $F$ is defined as follows. The vertical tangent bundle of $Z$ has a natural ${\text{spin}^c}$ structure, so let $S_v$ denote the spinor bundle 
of the vertical tangent bundle. Then let $F=S^{*,+}_v$, that is the dual of the even graded part of the ${\mathbb Z}_2$-graded bundle $S_v$.
\end{enumerate}
\end{definition}
\begin{remark}
	 {If $K$ is simply-connected, $G/K$ is $G$-equivariantly $\text{Spin}^c$ because $H_G^3(G/K, \mathbb{Z})\cong H_K^3(\text{pt}, \mathbb{Z})=0$. Thus $\Gamma_0(\text{Cliff}(TG/K))$ is Morita equivalent to $C_0(G/K)$. }
\end{remark}

\subsection{Equivalence with the analytic definition}
In this subsection, we sketch a proof that $KK^G_{geom}(\Gamma_0(\text{Cliff}(TX)), \Gamma_0( \tau_Y))$ is equivalent 
to the analytic definition.
Consider a twisted $G$-equivariant correspondence $\mathcal C$
 \begin{equation}
\xymatrix{& \ar[ddl]_b (M , \xi)\ar[ddr]^{f} &\\
&&\\
X && Y}
 \end{equation}
To a $\tau_Y$-twisted $G$-${\text{spin}^c}$ map $f:M \to Y$ as above, similar to a construction by Connes-Skandalis \cite{ConnesSkandalis84}, one can associate an element $f!: KK^G((M, {\rm Cliff}(TM)),(Y, \tau_Y))$ as
follows. The group $K_G^0(T^*M \times Y, \pi^*{\rm Cliff}(TM))$, where $\pi:T^*M\to M$ denotes the projection map, is generated by families of symbols $\sigma$ parametrised by $Y$, that are also 
 ${\rm Cliff}(TM)$-module maps. The quantization of $\sigma$ determines a family $\Psi^*(\sigma)$ of pseudodifferential operators parametrised by $Y$, and therefore an element 
$\Psi^*([\sigma]) \in KK^G(\Gamma_0({\rm Cliff}(TM)),\Gamma_0(\tau_Y))$, since $f^*(\tau_Y) {\otimes} \text{Cliff}(TM) \cong {\rm End}(S)$.
 Next we will argue that such a map $f$ determines a family of symbols 
$[\sigma_f] \in K_G^0(T^*M \times Y, \pi^*{\rm Cliff}(TM))$, and then define 
\begin{equation}
f! = \Psi^*([\sigma_f]) \in KK^G(\Gamma_0({\rm Cliff}(TM)), \Gamma_0(\tau_Y)).
\end{equation}

We now sketch the construction of $\sigma_f$. Since $f$ is twisted $G$-${\text{spin}^c}$, then there is a $G$-vector bundle $S$ over $M$ such that there is an equivariant isomorphism of algebra bundles over $M$,
\begin{equation}
f^*(\tau_Y) {\otimes} \text{Cliff}(TM) \cong {\rm End}(S).
\end{equation}
Let $\Omega_f$ denote the $G$-open subset of $M \times Y$ defined by $\Omega_f=\{(x,y) \in M\times Y| d(f(x), y)<1\}$. It is a tubular neighbourhood of the graph of $f$,
where we assume that the injectivity radius of $Y$ is equal to $1$. Let $pr_M: M\times Y \to M$ denote the projection, which is a $G$-map. Finally, define 
$\sigma_f (x,\xi,y)= M^{1/2}(x,y) ||\xi||^{-1}  c(\xi,0) + (1-M)^{1/2}(x,y)||y-f(x)||^{-1} c(0,y-f(x)) \in {\frak L}(S_x)$, for all $(x,y) \in \Omega_f,\, \xi \in T^*_xM\setminus \{0\}$, and
where $M(x,y)$ is a smooth $G$-function on $M\times Y$ such that $M(x,y)=1$ if $d(x,y)<\epsilon<1$ and $M(x,y)=0$  if $d(x,y)>1-\epsilon$. Then $(\Omega_f, p_M^*(S), \sigma)$ 
defines a Clifford $G$-symbol in the sense of \cite{ConnesSkandalis84}, such that $[\sigma_f] \in K_G^0(T^*M \times Y, \pi^*{\rm Cliff}(TM))$, completing the desired construction.

Given the $G$-correspondence as above, note that $b^*: \Gamma_0( {\rm Cliff}(TX)) \to \Gamma_0({\rm Cliff}(TM))$ is a homomorphism since $b$ is proper, which therefore
defines an element \begin{equation}b^* \in KK( \Gamma_0( {\rm Cliff}(TX)), \Gamma_0({\rm Cliff}(TM))).\end{equation}
The $G$-vector bundle $\xi \to M$ that is also a ${\rm Cliff}(TM)$-module, determines an element $[\xi] \in K_0^G(\Gamma_0({\rm Cliff}(TM)))$, as well as an element $[[\xi]] \in KK^G(\Gamma_0( {\rm Cliff}(TM)), \Gamma_0( {\rm Cliff}(TM)))$.
We can form 
\begin{equation}
[[\xi]]\otimes_{\Gamma_0( {\rm Cliff}(TM))} f! \in KK^G(\Gamma_0( {\rm Cliff}(TM)), \Gamma_0( \tau_Y)).
\end{equation}
Therefore
\begin{equation}
b^*([[\xi]]\otimes_{\Gamma_0( {\rm Cliff}(TM))} f!) \in KK^G(\Gamma_0(  {\rm Cliff}(TX)), \Gamma_0( \tau_Y)).
\end{equation}
 This completes the construction of the map,
\begin{equation}
\Phi: KK^G_{geom}(\Gamma_0(  {\rm Cliff}(TX)), \Gamma_0( \tau_Y)) \longrightarrow KK^G(\Gamma_0(  {\rm Cliff}(TX)), \Gamma_0( \tau_Y)).
\end{equation} 

By a standard method of proof, cf. \cite{EM,Wang, BMRS09}, one shows that this is an isomorphism of groups. We only give a brief outline of this isomorphism, as it is not central to our paper.

The surjectivity of $\Phi$ follows from the equivariant version of Proposition 5.3 in \cite{BMRS09}. 

Next the $C^*$-algebras 
$\Gamma_0(  {\rm Cliff}(TX)$ and $C_0(X)$ are $G$-Poincar\'e dual, therefore by \cite{EEK}
$$
PD: KK^G(\Gamma_0(  {\rm Cliff}(TX)), \Gamma_0( \tau_Y)) \cong KK^G(\CC, C_0(X) \otimes\Gamma_0( \tau_Y))
$$
and 
$$
PD_{geom}: KK^G_{geom}(\Gamma_0(  {\rm Cliff}(TX)), \Gamma_0( \tau_Y)) \cong KK^G_{geom}(\CC, C_0(X) \otimes\Gamma_0( \tau_Y)).
$$
On the other hand, an equivariant version of \cite{Wang} says that 
$$
\Phi': KK^G_{geom}(\CC, C_0(X) \otimes\Gamma_0( \tau_Y)) \cong KK^G(\CC, C_0(X) \otimes\Gamma_0( \tau_Y)),
$$
Let $\alpha \in KK^G_{geom}(\Gamma_0(  {\rm Cliff}(TX)), \Gamma_0( \tau_Y))$ be such that $\Phi(\alpha)=0$. Then 
$PD(\Phi(\alpha))=0$ so that $\Phi'^{-1}PD(\Phi(\alpha))=0$. But this is equal to $PD_{geom}(\alpha)=0$ which implies that
$\alpha=0$ proving injectivity of $\Phi$ and completing the sketch of proof.

\begin{appendix}
\section{$K$-theory of reduced group $C^*$-algebras}\label{kcstar}
 {In this appendix, we let $G$ be a real reductive linear Lie group. We would like to give a brief review of the results in \cite{W} and \cite{CCH} on the $K$-theory of $C_r^*G$, which reflects the tempered representations of $G$. The classification of the latter by parabolic induction can be found in \cite{K}. We also give an explicit description of the map $K_K^\bullet(K, \tau_K^K)\to K_\bullet(\Gamma_0(\tau_G^G)\rtimes G)$ in terms of weights when $G$ is a complex semisimple linear Lie group, based on the commutative diagram and the description of the ordinary Connes-Kasparov map in terms of weights.}

 {Given a Cartan decomposition $\mathfrak{g}=\mathfrak{k}\oplus\mathfrak{p}$ and a Cartan subalgebra $\mathfrak{h}$ of $\mathfrak{g}$, we define the following
\begin{enumerate}
	\item $\mathfrak{a}:=\mathfrak{h}\cap\mathfrak{p}$, 
	\item $A:=$the analytic subgroup of $G$ with Lie algebra $\mathfrak{a}$. 
	\item $\mathfrak{m}:=$the orthogonal complement of $\mathfrak{a}$ in $Z_\mathfrak{g}(\mathfrak{a})$. 
	\item $M_0:=$the analytic subgroup of $G$ with Lie algebra $\mathfrak{m}$, and 
	\item $M=Z_K(\mathfrak{a})M_0$. 
\end{enumerate}}
 {Using the abelian Lie subalgebra $\mathfrak{a}$ one can define a root space decomposition of $\mathfrak{g}$: 
\[\mathfrak{g}=\mathfrak{a}\oplus\bigoplus_{\alpha\in\Delta_{\mathfrak{a}^*}}\mathfrak{g}_\alpha.\]
With a choice of positive roots, let 
\[\mathfrak{n}:=\bigoplus_{\alpha\in\Delta_{\mathfrak{a}^*}^+}\mathfrak{g}_\alpha\]
and $N$ the analytic subgroup of $G$ with Lie algebra $\mathfrak{n}$. }
 {\begin{definition}
	The cuspidal parabolic subgroup of $G$ associated to $\mathfrak{h}$ and $\Delta_{\mathfrak{a}^*}^+$ is defined to be $P:=MAN$.
\end{definition}
Let $T$ be a maximal torus of $K$. Note that $\mathfrak{t}\cap \mathfrak{m}$ is a Cartan subalgebra of $\mathfrak{m}$ and so $M$ admits discrete series or limits of discrete series. 
\begin{theorem}[\cite{K}, Theorem 14.76] 
Every irreducible tempered representation $\pi$ of $G$ can be obtained by induction from a cuspidal parabolic subgroup $P=MAN$ as $\pi_{\sigma, \nu}:=\text{Ind}_P^G(\sigma\otimes e^{i\nu}\otimes 1)$ where $\sigma\in\widehat{M}_{\text{ds}}$ or a limit of discrete series, and $\nu\in\mathfrak{a}^*$. 
\end{theorem}}
 {\begin{definition}
	Let $W:=N_K(\mathfrak{a})/Z_K(\mathfrak{a})$. 
\end{definition}}
 {Since $N_K(\mathfrak{a})<N_K(M)$, $W$ also acts on $M$ by conjugation and thus also on $\widehat{M}_\text{ds}$. Let $W_\sigma:=\{w\in W| w\sigma=\sigma\}$. Let the unitary intertwining operator
\[A_P(w, \sigma, \nu): \text{Ind}_P^G(\sigma\otimes e^{i\nu}\otimes 1)\to \text{Ind}_P^G(w\sigma\otimes e^{iw\nu}\otimes 1)\]
be defined as in \cite[Theorem 7.22]{K}. 
\begin{definition}
	\[W'_\sigma:=\{w\in W_\sigma| A_P(w, \sigma, 0)\in\mathbb{C}I\}\]
\end{definition}}
 {One can write $W_\sigma=W'_\sigma\rtimes R_\sigma$ for some elementary abelian 2-subgroup $R_\sigma$. 
\begin{theorem}\label{reducedcstar}
Let $G$ be a real reductive linear Lie group.
\begin{enumerate}
	\item{\cite[Proposition 6.7]{CCH}}
	Let $\mathcal{E}_\sigma\to\mathfrak{a}^*$ be the $G$-Hilbert space bundle whose fiber at $\nu$ is $\text{Ind}_P^G(\sigma\otimes e^{i\nu}\otimes 1)$. Then we have the isomorphism
	\[C_r^*G\cong \bigoplus_{[P]}\bigoplus_{[\sigma]\in \widehat{M}_\text{ds}/W}\Gamma_0(\mathcal{K}(\mathcal{E}_\sigma))^{W_\sigma}\]
	where the first sum is taken over conjugacy classes of cuspidal parabolic subgroups $P$ and the second over representatives of $W$-orbits of discrete series representations. The isomorphism is realized by the map
	\[f\mapsto \int_G f(g)\pi_{\sigma, \nu}(g)dg\text{ at }\nu.\]
	\item{\cite{W}}\label{wassermann}The above $C^*$-algebras are Morita equivalent to 
	\[\bigoplus_{[P]}\bigoplus_{[\sigma]\in\widehat{M}_\text{ds}/W}C_0(\mathfrak{a}^*/W'_\sigma)\rtimes R_\sigma.\]
	We have that
	\[K_{\text{dim }G/K}(C_r^*G)\cong\bigoplus_{[P]}\bigoplus_{\substack{[\sigma]\in\widehat{M}_\text{ds}/W\\ W'_\sigma=1}}\mathbb{Z}\]
	and $K_{\text{dim }G/K+1}(C_r^*(G))=0$. Moreover, the Connes-Kasparov Conjecture is true, i.e. the Dirac induction map 
	\[R(K)\to K_\bullet(C_r^*G)\]
	is an isomorphism. 
\end{enumerate}
\end{theorem}
The Connes-Kasparov conjecture asserts that the Dirac induction}
 {\begin{align*}
	R(K)&\to K_\bullet(C_r^*G)\\
	[V_\mu]&\mapsto \text{Ind}(\slashed{\partial}_{G/K}\otimes V_\mu), 
\end{align*}
with $\slashed{\partial}_{G/K}\otimes V_\mu$ acting on the space of sections $L^2(G/K, G\times_K(S\otimes V_\mu))$, {is an isomorphism}. Here $S$ is a $K$-representation induced by the homomorphism $K\to \text{Spin}(\mathfrak{p})$ which is a lifting of the adjoint representation $K\to SO(\mathfrak{p})$, followed by the the spin representation $\text{Spin}(\mathfrak{p})\to \text{End}(S)$. The Connes-Kasparov conjecture is proved in \cite{W} in the case of connected real reductive linear Lie groups. }
\begin{example}[\cite{PP}]  {When $G$ is a complex connected semisimple Lie group, there is only one conjugacy class of cuspidal parabolic subgroups, and we can take $M$ to be a maximal torus of $K$, $\mathfrak{a}=i\mathfrak{m}$, and $W$ the Weyl group of $K$. By \cite[Proposition 4.1]{PP}, $C_r^*G$ is Morita equivalent to $C_0(\widehat{M}\times\mathfrak{a}^*/W)$ which indeed is Morita equivalent to 
\[\bigoplus_{[\sigma]\in\widehat{M}/W}C_0(\mathfrak{a}^*/W_\sigma)\]
by Theorem \ref{reducedcstar} (\ref{wassermann}) (here $R_\sigma=1$ for all $\sigma$). The summand $C_0(\mathfrak{a}^*/W_\sigma)$ contributes a copy of $\mathbb{Z}$ to $K_\bullet(C_r^*G)$ if and only if $W_\sigma=1$, i.e. $\sigma$ is a regular weight. By fixing a set of dominant weights $\Lambda^*_+(K)$ in $\widehat{M}$ and picking from $\Lambda_+^*(K)$ representatives of the $W$-orbits of $\widehat{M}$, we see that those $\sigma$ is of the form $\sigma=\mu+\rho$ for $\mu$ a dominant weight and $\rho$ the half sum of the positive roots of $K$. We may view $K_\bullet(C_r^*G)$ as the free abelian group generated by regular dominant weights. By \cite[Section 5]{PP}, one can decompose $L^2(G/K, G\times_K(S\otimes V_\mu))\cong L^2(G, G\times(S\otimes V_\mu))^K$ as 
\[\int_{\widehat{M}\times\mathfrak{a}^*/W}\mathcal{H}_{\sigma, \nu}\otimes(\mathcal{H}_{\sigma, \nu}^*\otimes S\otimes V_\mu)^K d(\sigma, \nu)\]
viewed as a Hilbert space bundle over $\widehat{M}\times\mathfrak{a}^*/W$, where $\mathcal{H}_{\sigma, \nu}$ is the underlying Hilbert space for $\pi_{\sigma, \nu}$. Then $\slashed{\partial}_{G/K}\otimes V_\mu$ acts as $\text{Id}\otimes\varphi(\sigma, \nu, \mu)$ for some $\varphi(\sigma, \nu, \mu)\in\text{End}((\mathcal{H}_{\sigma, \nu}\otimes S\otimes V_\mu)^K)$. We then have that}
 {\begin{enumerate}
	\item If $\mu+\rho-\sigma$ is in the positive Weyl chamber and not 0, then $\varphi(\sigma, \nu, \mu)$ is invertible for all $\nu\in\mathfrak{a}^*$.
	\item If $\mu+\rho-\sigma$ is not in the positive Weyl chamber, then $(\mathcal{H}_{\sigma, \nu}\otimes S\otimes V_\mu)^K=0$. 
	\item If $\mu+\rho=\sigma$ then $\varphi(\sigma, \nu, \mu)$ acts by Clifford multiplication by $\nu$. 
\end{enumerate} 
Hence $\text{Ind}(\slashed{\partial}_{G/K}\otimes V_\mu)$ can be thought of as the complex of vector bundles over $\{\mu+\rho\}\times\mathfrak{a}^*$
\[0\longrightarrow(\mathcal{H}^*_{\mu+\rho, \nu}\otimes S\otimes V_\mu)^K\stackrel{\text{Cl}(\nu)}{\longrightarrow}(\mathcal{H}^*_{\mu+\rho, \nu}\otimes S\otimes V_\mu)^K\longrightarrow 0\]
whose $K$-theory class is the Bott element $\beta_{\mu+\rho}$ of the summand $K_\bullet(C_0(\mathfrak{a}^*/W_{\mu+\rho}))\cong K_\bullet(C_0(\mathfrak{a}^*))$ of $K_\bullet(C_r^*G)$. In short, the Dirac induction map is given by $\rho$-shift, i.e. 
\[[V_\mu]\mapsto \beta_{\mu+\rho}.\]}

 {From the commutative diagram (\ref{commdiag2}), we can see that $K_\bullet(\Gamma_0(\tau_G^G)\rtimes G)$ is isomorphic to the free abelian group generated by the set of regular level $k$ weights $\Lambda^*_{k, \text{reg}}(K)$. The map 
\[R^{k-\textsf{h}^\vee}(K)\cong K_K^\bullet(K, \tau_K^K)\to K_\bullet(\Gamma_0(\tau_G^G)\rtimes G)\]
then sends the class of positive energy representation $W_\mu$ with the lowest weight $\mu\in \Lambda^*_{k-\textsf{h}^\vee}(K)$ to the generator of $K_\bullet(\Gamma_0(\tau_G^G)\rtimes G)$ labelled by $\mu+\rho\in \Lambda_{k, \text{reg}}^*(K)$.}
\end{example}

 {The Freed-Hopkins-Teleman Theorem, Theorem \ref{reducedcstar} and the commutative diagram (\ref{commdiag2}) inspire the following conjecture.
\begin{conjecture}\label{loopinduction}
	\begin{enumerate}
		\item The $K$-group $K_\bullet(\Gamma_0(\tau_G^G)\rtimes G)$ reflects the tempered, positive energy representations of $LG$ in a suitable sense.
		\item The map $q: K_K^\bullet(K, \tau_K^K)\to K_\bullet(\Gamma_0(\tau_G^G)\rtimes G)$ is given by the Dirac induction $[W_\mu]\mapsto \text{Ind}(\slashed{\partial}_{LG/K}\otimes V_\mu)$ for $\mu\in \Lambda^*_{k-\textsf{h}^\vee}(K)$.
	\end{enumerate}
\end{conjecture}}

\section{Quantization of $q$-Hamiltonian $G$-spaces}

$q$-Hamiltonian manifolds, which are defined in \cite{AMM} and equipped with Lie group-valued moment maps as opposed to ordinary moment maps which take values in dual of Lie algebras, can be thought of as the `exponentiated' version of ordinary Hamiltonian manifolds, and gives rise to finite dimensional construction of moduli spaces of Riemann surfaces. In \cite{M1, M2}, quantization of 
$q$-Hamiltonian manifolds with Hamiltonian action by compact Lie groups is defined as pushforward of twisted $K$-theory fundamental class of the manifold along the group-valued moment map. In this appendix, we will construct $q$-Hamiltonian $G$-manifolds, where $G$ is noncompact acting properly and cocompactly,  from $q$-Hamiltonian $K$-manifolds, where $K$ is a maximal compact subgroup of $G$, by means of induction and cross-section, following \cite{H, AMM} and define their quantization in terms of twisted $G$-equivariant $KK$-theory, which we show is a finite dimensional Frobenius ring, hence associated to a TQFT.

\begin{definition}
	We say $g\in G$ is strongly stable if its centralizer $G_g$ is compact. Let $\mathcal{D}$ be the set of strongly stable elements in $G$.
\end{definition}
\begin{remark}
	In \cite{W}, strongly stable elements refer to those in the Lie algebra $\mathfrak{g}$ with compact centralizers. 
\end{remark}
By adapting \cite[Corollary 2.4, Propositions 2.6 and 2.8]{W}, we have the following.
\begin{proposition}
	\begin{enumerate}
		\item $\mathcal{D}$ is an open set of $G$. 
		\item $\mathcal{D}$ is non-empty if and only if $\text{rk }G=\text{rk }K$.
		\item Let $\mathcal{D}$ be non-empty and $T$ a maximal torus of $K$ (so $T$ is a compact Cartan subgroup of $G$). If $\mathcal{F}:=\{\xi\in\mathfrak{t}| \alpha(\xi)\neq 2k\pi i\text{ for }k\in\mathbb{Z}\text{ and }\alpha\text{ a noncompact root}\}$, then $\mathcal{D}=\text{Ad}_g\text{exp}(\mathcal{F})$. 
	\end{enumerate}
\end{proposition}

\begin{proposition}
	Let $\langle, \rangle_\mathfrak{g}$ be the normalized Killing form on $\mathfrak{g}$ such that its restriction to $\mathfrak{k}$ is the basic inner product of $\mathfrak{k}$, and $\eta_G=\frac{1}{12}\langle\theta^L, [\theta^L, \theta^L]\rangle_\mathfrak{g}$. Let $\mu_{G/K, k_0}: G/K\to G$ be the map
	\[gK\mapsto \text{Ad}_g k_0\]
	for some $k_0\in K$ and $\omega_{G/K, k_0}$ the 2-form on $G/K$ defined by 
	\[\omega_{G/K, k_0}(X_\xi, X_\zeta)=\frac{1}{2}(\langle\text{Ad}_{k_0}\xi, \zeta\rangle_\mathfrak{g}-\langle\xi, \text{Ad}_{k_0}\zeta\rangle_\mathfrak{g}).\]
	Here $X_\xi$ is the vector field on $G/K$ induced by the infinitesimal action by $\xi$ through left translation. Then $(G/K, \omega_{G/K, k_0}, \mu_{G/K, k_0})$ is a $q$-Hamiltonian $G$-manifold. 
	If we further assume that $k_0\in\mathcal{D}\cap K$, then $\mu_{G/K, k_0}$ is a proper map and its image, the conjugacy class containing $k_0$, is acted properly upon by $G$ through conjugation.
\end{proposition}
\begin{proof}
	This is just a straightforward adaptation of \cite[Proposition 3.1]{AMM} on conjugacy classes of $K$ as an example of $q$-Hamiltonian $K$-spaces. 
\end{proof}

Let $k_0\in K$ and $\mathfrak{p}=\mathfrak{k}^\perp$ with respect to $\langle, \rangle_\mathfrak{g}$ so that $\mathfrak{g}=\mathfrak{k}\oplus \mathfrak{p}$ is the Cartan decomposition of $\mathfrak{g}$. Then $\mathfrak{p}$ and $G/K$ are $K$-equivariantly diffeomorphic through the map 
\[\xi\mapsto \text{exp}(\xi)K.\] 
In this way, $\omega_{G/K, k_0}\in \Omega^2(G/K)$ corresponds to the 2-form $\omega_{\mathfrak{p}, k_0}\in \Omega^2(\mathfrak{p})$ defined by 
\[\omega_{\mathfrak{p}, k_0}(X, Y)=\frac{1}{2}(\langle\text{Ad}_{k_0}X, Y\rangle_\mathfrak{g}-\langle X, \text{Ad}_{k_0}Y\rangle_\mathfrak{g}).\]
The moment map $\mu_{G/K, k_0}$ corresponds to 
\begin{align*}
	\mu_{\mathfrak{p}, k_0}:\mathfrak{p}&\to G\\
	\xi&\mapsto \text{Ad}_{\text{exp}(\xi)}k_0.
\end{align*}
Thus $(\mathfrak{p}, \omega_{\mathfrak{p}, k_0}, \mu_{\mathfrak{p}, k_0})$ is a $q$-Hamiltonian $G$-space.

\begin{proposition}\label{inflation}
	Let $N$ be a $K$-manifold, and $M$ the \emph{induction of }$N$, i.e. the quotient
	\[M:=G\times_K N:=G\times N/((g, n)\sim (gk^{-1}, k\cdot n), k\in K)\]
	with the $G$-action $h\cdot[g, n]=[hg, n]$. Then 
	\begin{enumerate}
		\item\label{untwist} The map 
		\begin{align*}
			s: \mathfrak{p}\times N&\to M\\
			(\xi, n)&\mapsto [\text{exp}(\xi), n]
		\end{align*}
		is a $K$-equivariant diffeomorphism. 
		\item\label{tangentid} $T_{[e_G, n]}M$ can be identified with $T_nN\oplus\mathfrak{p}$. Under this identification, $\xi\in\mathfrak{p}$ corresponds to $(X_\xi)_{[e_G, n]}\in T_{[e_G, N]}M$. 
	\end{enumerate}
\end{proposition}
\begin{proof}
	It is easy to check that, if for $g\in G$, its Cartan decomposition is $g=\text{exp}(\xi)k$ for $\xi\in\mathfrak{p}$, $k\in K$, then the map
	\begin{align*}
		t: M&\to \mathfrak{p}\times N\\
		[g, n]&\mapsto (\xi, k\cdot n)
	\end{align*}
	is smooth, $K$-equivariant and the inverse of $s$. This proves (\ref{untwist}). (\ref{tangentid}) follows from (\ref{untwist}). 
\end{proof}

\begin{proposition}\label{qHaminduc}
	\begin{enumerate}
		\item Let $(N, \omega_N, \mu_N)$ be a $q$-Hamiltonian $K$-manifold. On the induction $M$ of $N$, define the 2-form $\omega_M$ by requiring that it be $G$-invariant and 
	\[(\omega_M)_{[e_G, n]}(v+\xi, w+\zeta)=\omega_N(v, w)+\frac{1}{2}(\langle\text{Ad}_{\mu_N(n)}\xi, \zeta\rangle_\mathfrak{g}-\langle \xi, \text{Ad}_{\mu_N(n)}\zeta\rangle_\mathfrak{g})\]
	(Here we use the identification of tangent spaces as in Proposition \ref{inflation} (\ref{tangentid})). Also define 
	\begin{align*}
		\mu_M: M&\to G\\
		[g, n]&\mapsto \text{Ad}_g{\mu_N(n)}.
	\end{align*}
	Then $(M, \omega_M, \mu_M)$ is a $q$-Hamiltonian $G$-manifold. 
		\item If we further assume that $G$ and $K$ be of equal rank such that $\text{Im }\mu_N\subseteq\mathcal{D}\cap K$, then $G$ acts on $M$ properly.
	\end{enumerate}
\end{proposition}

\begin{remark}
	The construction of the $q$-Hamiltonian $G$-manifold $(M, \omega_M, \mu_M)$ is inspired by the similar construction of Hamiltonian structure on $M$ from that on $N$ in \cite[\S 12.2]{H}. 
\end{remark}
\begin{proof}
	It suffices to check that $(M, \omega_M, \mu_M)$ satisfies the three conditions in the definition of $q$-Hamiltonian spaces (\cite[Definition 2.2]{AMM}). 
	\begin{enumerate}
		\item $d\omega_M=-\mu_M^*\eta_G$. As both $\omega_M$ and $\eta_G$ are $G$-invariant, and $N$ as the submanifold of $M$ through the embedding
		\begin{align*}
			i: N&\to M\\
			n&\mapsto [e_G, n]
		\end{align*}
		is a slice of $M$, it suffices to check that the equation holds on $T_{[e_G, n]}M\cong T_nN\oplus \mathfrak{p}$ for $n\in N$. Note that 
		\begin{align*}
			i^*d\omega_M&=di^*\omega_M\\
						&=d\omega_N\\
						&=-\mu_N^*\eta_G\\
						&=-(\mu_M\circ i)^*\eta_G\\
						&=-i^*\mu_M^*\eta_G
		\end{align*}
		Thus the equation holds when restricted to $T_n N$. Next, $\mathfrak{p}\subseteq T_{[e_G, n]}M$ and $\omega_M$ restricted to it can be identified naturally with $(\mathfrak{p}, \omega_{\mathfrak{p}, \mu_N(n)}, \mu_{\mathfrak{p}, \mu_N(n)})$, and the equation restricted to $\mathfrak{p}$ is equivalent to, under this identification, $d\omega_{\mathfrak{p}, \mu_N(n)}=-\mu_{\mathfrak{p}, \mu_N(n)}^*\eta_G$, which is true as we have shown that $(\mathfrak{p}, \omega_{\mathfrak{p}, \mu_N(n)}, \mu_{\mathfrak{p}, \mu_N(n)})$ is a $q$-Hamiltonian $G$-manifold. 
		\item $\displaystyle\iota_{X_\xi}\omega_M=\frac{1}{2}\mu_M^*\langle\theta^L+\theta^R, \xi\rangle_\mathfrak{g}$. Again, by $G$-invariance of $\omega_M$, $X_\xi$ and $\displaystyle \frac{1}{2}\mu_M^*\langle\theta^L+\theta^R, \xi\rangle_\mathfrak{g}$, it suffices to show that the equation holds on $T_{[e_G, n]}M\cong T_nN\oplus \mathfrak{p}$. If $\alpha\in\mathfrak{g}$ and $\alpha=\xi+\eta$ for $\xi\in\mathfrak{p}$ and $\eta\in\mathfrak{k}$, then $X_\alpha=X_\xi+X_\eta$ by the Baker-Campbell-Hausdorff formula. Besides, under the identification $T_{[e_G, n]}M\cong T_nN\oplus \mathfrak{p}$, we have $(X_\xi)_{[e_G, n]}\in T_{[e_G, n]}M$ corresponds to $\xi\in\mathfrak{p}$ and $(X_\eta)_{[e_G, n]}\in T_n N$. We then have
		\begin{align*}
			(\iota_{X_\alpha}\omega_M)_{[e_G, n]}(w+\zeta)&=(\omega_M)_{[e_G, n]}(X_\eta+\xi, w+\zeta)\\
												&=(\omega_N)_n(X_\eta, w)+\frac{1}{2}(\langle\text{Ad}_{\mu_N(n)}\xi, \zeta\rangle_\mathfrak{g}-\langle\xi, \text{Ad}_{\mu_N(n)}\zeta\rangle_\mathfrak{g})\\
												&=\frac{1}{2}\mu_N^*(\theta^L+\theta^R, \eta)(w)+\frac{1}{2}(\theta^L+\theta^R, \xi)(\zeta)\ (\text{by \cite[Proposition 3.1]{AMM}})\\
												&=\frac{1}{2}\mu_M^*(\theta^L+\theta^R, \xi)(w+\zeta)\\
												&(\text{As }(\mu_{M*})_{[e_G, n]}(w+\zeta)=(\mu_{N*})_n(w)+\zeta\text{ because }\mu_M([e_G, n])=\mu_N(n))
		\end{align*}
		\item $\text{ker}(\omega_M)_{[g, n]}=\{X_\xi| \xi\in\text{ker}(\text{Ad}_{\mu_M([g, n])}+1)\}$. It suffices to prove the equation for $g=e_G$ because both sides of the equations are $G$-invariant: that the LHS is $G$-invariant is due to the $G$-invariance of $\omega_M$. For the $G$-invariance of the RHS, note that $g\cdot(X_\xi)_{[e_G, n]}=(X_{\text{Ad}_g\xi})_{[g, n]}$ and $\xi\in \text{ker}(\text{Ad}_{\mu_M([e_G, n])}+1)$ if and only if $\text{Ad}_g\xi\in\text{ker}(\text{Ad}_{\mu_M([g, n])}+1)$ by the $G$-equivariance of $\mu_M$. Now for $v\in T_nN$ and $\xi\in\mathfrak{p}$, $v+\xi\in\text{ker}(\omega_M)_{[e_G, n]}$ if and only if 
		\begin{enumerate}
			\item $v\in\text{ker}(\omega_N)_n$, i.e. $v=(X_\eta)_n$ for some $\eta\in\mathfrak{k}$ such that $\eta\in\text{ker}(\text{Ad}_{\mu_N(n)}+1)_\mathfrak{k}$, and 
			\item $X_\xi\in\text{ker}(\omega_{\mathfrak{p}, \mu_N(n)})$, i.e. $\xi\in\text{ker}(\text{Ad}_{\mu_N(n)}+1)_\mathfrak{g}$ (or equivalently $\xi\in\text{ker}(\text{Ad}_{\mu_N(n)}+1)_\mathfrak{p}$ as $\xi\in\mathfrak{p}$). 
		\end{enumerate}
		Moreover, $v+X_\xi=X_\eta+X_\xi=X_{\eta+\xi}$ and $\eta+\xi\in\text{ker}(\text{Ad}_{\mu_N(n)}+1)_\mathfrak{g}$. This completes the proof of the equation. 
	\end{enumerate}
	The second part of the proposition follows by adapting the proof of \cite[Lemma 6.12]{H}.
\end{proof}

One can also construct a $q$-Hamiltonian $K$-manifold from a $q$-Hamiltonian $G$-manifold by means of taking the so-called $q$-Hamiltonian \emph{cross-section}, defined to be $N:=\mu_M^{-1}(K)$. The following proposition, which asserts that induction and taking cross-section are reverse to each other, is straightforward (cf. \cite[\S 12.4]{H}).

\begin{proposition}\label{cross-section}
	Let $G$ and $K$ have the same rank, $(M, \omega_M, \mu_M)$ be a $q$-Hamiltonian $G$-manifold with proper $G$-action and $\text{Im}\mu_M\in \mathcal{D}$. Suppose the map $\mu_M: M\to G$ composed with the quotient map $G\to G/K$ has $e_GK$ as a regular value. Define the $q$-Hamiltonian cross-section $N:=\mu_M^{-1}(K)$ and let $i: N\to M$ be the embedding. Then $(N, i^*\omega_M, \mu_M\circ i)$ is a $q$-Hamiltonian $K$-manifold. 
\end{proposition}

\bigskip

Let $M$ be a  $q$-Hamiltonian proper, cocompact $G$-manifold. By Abel's global slice theorem \cite{Abels}, 
  $M$ is $G$-equivariantly diffeomorphic to $G\times_K N$, therefore there is a projection map $b: M \to G/K$. 
  Since there is a canonical twisted $G$-${\text{spin}^c}$ structure on $M$, there is also a morphism $f: M \to (G, \tau^G_G)$ which is the $G$-valued moment
  map (denoted earlier by $\mu_M$), 
  
   \begin{definition}
The quantization of $M$, 
  $\cQ(M)$, is defined to be the equivalence class of the correspondence,
   \begin{equation}
\xymatrix{& \ar[ddl]_b (M , \xi)\ar[ddr]^{f} &\\
&&\\
G/K && (G, \tau^G_G)}
 \end{equation}
where $\xi= G\times_K TN$. That is, $\cQ(M) \in KK^G(G/K, (G, \tau_G^G))$.
\end{definition}

Define ${\mathcal D}_\mu = G\times_K {\mathcal C}_\mu$, on which $G$ acts by left translation on the first factor and trivially on the second factor. 
Here $\mu\in\Lambda^*_{k-\textsf{h}^\vee}(K)$ is a level $k-\textsf{h}^\vee$ weight and $\mathcal{C}_\mu$ the conjugacy class in $K$ containing $\text{exp}\left(\frac{B^\sharp(\mu)}{k-\textsf{h}^\vee}\right)$.
  Note that ${\mathcal D}_\mu$ has a proper $G$-action and is also $G$-cocompact. 
  
  In fact, we show in Proposition \ref{qHaminduc}  that  ${\mathcal D}_\mu$  is a 
 $ q$-Hamiltonian $G$-space with these properties.
Then there is a natural projection map 
$b: {\mathcal D}_\mu \longrightarrow G/K$. 
There is also a natural twisted $G$-${\text{spin}^c}$ map $f: {\mathcal D}_\mu \longrightarrow G$ defined by,  
$[g, c] \longrightarrow gcg^{-1}$, which is the $G$-valued moment
  map (denoted earlier by $\mu_{{\mathcal D}_\mu}$), . 

 \begin{equation}\label{Kgeneratorcorrespondence}
\xymatrix{& \ar[ddl]_b ({\mathcal D}_\mu , \xi_\mu)\ar[ddr]^{f} &\\
&&\\
G/K && (G, \tau^G_G)}
 \end{equation}
where $\xi_\mu = G\times_K T{\mathcal C}_\mu$, are correspondences. 

\begin{theorem}
	The correspondences (\ref{generatorcorrespondence}) for $\mu\in\Lambda_{k-\textsf{h}^\vee}(K)$ are all the generators of 
	$KK^G(G/K, (G, \tau_G^G)),$ that is,
	 $$\cQ({\mathcal D}_\mu ) \in KK^G(G/K, (G, \tau_G^G)),$$
	 forms a complete set of generators.
	\end{theorem}

 Recall that $\cI_K$ denotes the Verlinde ideal for $(K, \tau^K_K)$, 
 and that  $$K_*^K(K, \tau_K^K) \cong R(K)/\cI_K.$$ Recall also that {$\cI_G =  \text{D-Ind} (\cI_K)$} is defined to be the Verlinde ideal for 
$G$, in which case $$KK^G_\bullet(G/K, (G, \tau_G^G)) \cong K_*(C^*_rG)/\cI_G,$$
showing that $KK^G_\bullet(G/K, (G, \tau_G^G))$ is a {\em Verlinde ring}.

Now $ K_0(C^*_rG)$  is a ring by Dirac induction, so this gives an alternate description of the ring structure of $KK^G_\bullet(G/K, (G, \tau_G^G))=K_*(C^*_rG)/\cI_G$ which 
has the following properties.

\begin{enumerate}
\item $K_*(C^*_rG)/\cI_G$ is a unital ring with involution.
\item $K_*(C^*_rG)/\cI_G$ has a finite $\bbZ$-basis, $K_*(C^*_rG)/\cI_G= \bbZ[Q]$ where $Q$ are the level $(k-h^\vee)$-weights in the compact case, and in the non-compact case 
they can be viewed as those Harishchandra parameters that are in the image of the  level $(k-h^\vee)$-weights under the Dirac induction isomorphism (see also Example A.6).
They are generated by a certain subset of discrete series of $M$ where $M$ appears in the decomposition $P = MAN$ (see Definition A.1).
\item $K_*(C^*_rG)/\cI_G$ has a trace map, $ \tr \colon K_*(C^*_rG)/\cI_G \longrightarrow \bbZ $, defined as the coefficient of the unit in  $K_*(C^*_rG)/\cI_G$.
\end{enumerate}

Therefore  $K_*(C^*_rG)/\cI_G$ is a finite dimensional {\em Frobenius ring}, hence it is associated to a TQFT.

\end{appendix}

\end{document}